\documentclass[11pt]{article}
    \usepackage{url}
    \usepackage{verbatim}
    \usepackage[titletoc]{appendix}
    \usepackage{graphicx}
	\usepackage{amsmath}
	\usepackage{amsthm}
	\usepackage{amsfonts}
	\usepackage{graphicx}
    \usepackage{nicefrac}
    \usepackage{longtable}
    \usepackage{color}
    \usepackage{graphicx, amssymb, subcaption,graphics}

    \newcommand{\hf}{\frac{1}{2}}

    \newcommand{\Dh}{\Delta_h}

    \newcommand{\nrm}[1]{\left\| #1 \right\|}

     \newcommand{\msfT}{\mathsf{T}}

    \newcommand\dt {{\Delta t}}

    \def\x{\mbox{\boldmath $x$}}
    \def\0{\mbox{\boldmath $0$}}
    
	\newtheorem{thm}{Theorem}[section]
	\newtheorem{prop}[thm]{Proposition}
	\newtheorem{cor}[thm]{Corollary}
	
	\newtheorem{lem}[thm]{Lemma}
	\newtheorem{rem}[thm]{Remark}

\begin{document}
	\title{A refined convergence estimate for a fourth order finite difference numerical scheme to the Cahn-Hilliard equation}
	
		\author{
Jing Guo\thanks{School of Software Engineering, South China University of Technology, Guangzhou, Guangdong 510006, 
P. R. China ({\tt 
z7198185@gmail.com})}
        \and
Cheng Wang\thanks{Mathematics Department, The University of 
Massachusetts, North Dartmouth, MA 02747, USA 
({\tt cwang1@umassd.edu})}
         \and
Yue Yan\thanks{School of Mathematics, Shanghai University of 
Finance and Economics, Shanghai 200433, P.R. China ({\tt corresponding author: yan.yue@mail.shufe.edu.cn})}
         \and
Xingye Yue\thanks{School of Mathematical Sciences, 
Soochow University, Suzhou, Jiangsu 215006, 
P. R. China ({\tt xyyue@suda.edu.cn})}
	}

	\maketitle
	\numberwithin{equation}{section}

	\begin{abstract}

In this article we present a refined convergence analysis for a second order accurate in time, fourth order finite difference numerical scheme for the 3-D Cahn-Hilliard equation, with an improved convergence constant. A modified backward differentiation formula temporal discretization is applied, and a Douglas-Dupont artificial regularization is included to ensure the energy stability. 
In fact, a standard application of discrete Gronwall inequality leads to a convergence constant dependent on the interface width parameter in an exponential singular form. We aim to obtain an improved estimate, with such a singular dependence only in a polynomial order. A uniform in time functional bounds of the numerical solution, including the higher order Sobolev norms, as well as the associated bounds for the first and second order temporal difference stencil, have to be carefully established. Certain recursive analysis has to be applied in the analysis for the BDF-style temporal stencil. As a result, we are able to apply a spectrum estimate for the linearized Cahn-Hilliard operator, and this technique leads to the refined error estimate. A three-dimensional numerical example of accuracy check is presented as well.
	\end{abstract}
		

\medskip
	

	\section{Introduction}
In this article we consider the Cahn-Hilliard model. For any $\phi \in H^1 (\Omega)$, with $\Omega \subset R^d$ ($d=2$ or $d=3$), the energy functional is given by (see \cite{cahn58} for a detailed derivation): 
	\begin{equation}
	\label{AC energy}
E(\phi)=\int_{\Omega}\left( \varepsilon^{-1} ( \frac{1}{4} \phi^4-\frac{1}{2} \phi^2 ) + \frac{\varepsilon}{2}|\nabla \phi|^2\right) d {\bf x} , 
	\end{equation} 
in which the parameter $\varepsilon$ corresponds to the diffuse interface width.  In turn, the Cahn-Hilliard equation becomes an $H^{-1}$ conserved gradient flow of the energy functional~\eqref{AC energy}: 
	\begin{equation}
	\label{CH equation}
\phi_t=\Delta \mu , \quad   \mu := \delta_\phi E = \varepsilon^{-1} ( \phi^3 - \phi ) - \varepsilon \Delta \phi , 
	\end{equation}
where $\mu$ is the chemical potential. Periodic boundary conditions are imposed for both the phase field, $\phi$, and the chemical potential, $\mu$, as well as the higher order derivatives.  Subsequently, the energy dissipation law follows from an inner product with~\eqref{CH equation} by $\mu$: $ E' (t)=-\int_{\Omega} |\nabla \mu |^2 d {\bf x} \le 0$. Of course, since the dynamical equation is constructed as an $H^{-1}$ gradient flow, the mass conservative identity is always valid:  $\int_\Omega \partial_t \phi  d{\bf x} = 0$. 


Most existing numerical works for the Cahn-Hilliard equation have been focused on either the second order finite difference or linear/quadratic polynomial finite element spatial discretization. 
Meanwhile, a fourth order and even more accurate spatial approximation would be suitable to capture the more detailed structure with a reduced computational cost. Among the spatially higher order numerical algorithm for the Cahn-Hilliard equation, the spectral/pseudo-spectral approximation is worthy of investigation; see the related references~\cite{cheng2022a, cheng16a, LiD2016b, LiX21a, LiX23, LiX24}, etc. On the other hand, the spectral/pseudo-spectral differentiation corresponds to a global operator in space, and the associated numerical implementation requires an $O (N^d \ln N)$ float point calculations, instead of $O (N^d)$ scale for the finite difference ones. In contrast, a fourth order finite difference differentiation turns out to be a good choice to balance the high order spatial accuracy and the computational cost. 

A second order accurate in time, fourth order long-stencil finite difference numerical scheme has been recently proposed and analyzed for the Cahn-Hilliard equation~\cite{cheng2019a}. In the temporal approximation, a second order backward differentiation formula (BDF) stencil is applied, combined with a second order extrapolation formula applied to the concave diffusion term. Other than these standard discretization, a second order artificial Douglas-Dupont regularization term is included in the numerical algorithm to ensure the energy stability at a theoretical level, following similar ideas as~\cite{chen12, chen2019b, cheng2019c, fengW18b, Hao2021, LiW18, Meng2020, yan18}, etc.  In addition, an optimal rate convergence analysis is derived for the numerical scheme, with second order temporal accuracy and fourth order spatial accuracy. 

On the other hand, as always observed in the standard error estimate, an application of discrete Gronwall inequality leads to a convergence constant dependent on the final time $T$ and on the interface parameter $\varepsilon$ in a singular and exponential way. In more details, its dependence on $T$ and $\varepsilon$ is usually in a form of $\exp ( C \varepsilon^{-m} T)$, with $m$ an integer. To address this subtle issue, a few theoretical efforts have been reported to improve the convergence constant for the Allen-Cahn and Cahn-Hilliard equations, such as \cite{feng04} for a first-order in time, fully discrete finite element scheme, \cite{feng15, feng16} for a first-order in time, discontinuous Galerkin scheme. In these theoretical works, the convergence constant is of order $O(e^{C_0 T} \varepsilon^{-m_0})$, for some positive integer $m_0$ and a constant $C_0$ independent on $\varepsilon$, instead of the singularly $\varepsilon$-dependent exponential growth. This improvement was based on a subtle spectrum analysis for the linearized Cahn-Hilliard operator (with certain given structure assumptions of the solution), provided in earlier PDE analysis literatures~\cite{alikakos94, alikakos93, chenx94, chenx96, chenx98}, etc. 

Meanwhile, most existing works of improved convergence constant have been focused on the numerical scheme with first order accurate temporal discretization. A similar analysis has been performed for a semi-discrete, second order accurate in time numerical method in a more recent work~\cite{guo21}, in which a modified Crank-Nicolson temporal approximation is taken, and the space is kept continuous in the analysis. In this article, we provide a refined convergence and error analysis for the second order accurate in time (modified BDF2 temporal discretization), fourth order long stencil difference numerical scheme for the 3-D Cahn-Hilliard equation, as proposed in~\cite{cheng2019a}. 

The energy stability of the proposed numerical scheme has been proved in the existing work, which in turn gives a discrete $H^1$ bound of the numerical solution. Meanwhile, such a bound is not sufficient to derive a refined error estimate. To overcome the difficulty associated with the discrete Gronwall inequality, we have to analyze the numerical errors, which come from the nonlinear part, the concave linear part and the surface diffusion part, in a unified way. The spectrum estimate for the linearized Cahn-Hilliard operator enables one to derive a refined estimate for the combined error. To achieve such a goal, a uniform in time (discrete) $H^m$ (for $m \ge 2$) bound for the numerical solution is needed, in addition to the $H^1$ bound from the energy stability estimate. Such a bound could be theoretically accomplished by taking a discrete inner product with $(-\Delta_{h, (4)})^m \phi^{n+1}$  (with $\phi$ the numerical solution, $\Delta_{h, (4)}$, $\nabla_{h, (4)}$ the long stencil difference approximation), combined with repeatedly application of discrete H\"older inequality and Sobolev estimates. In particular, the associated estimates in the fourth order long-stencil difference approximation have to be carefully derived, with the help of the Fourier analysis and the eigenvalue analysis. Moreover, all these bounds are dependent on the initial $H^m$ data, $\frac{1}{\varepsilon}$ (in a polynomial form), and independent on $T$. Since this bound is valid for any $m \ge 2$, a further observation indicates that $\| \phi^{n+1} - \phi^n \|_{H^k}$ (for an integer $k$) is always of order $O (\dt)$, and the constant is independent on $T$. In turn, a uniform in time $O (\dt)$ estimate for $\| \phi^{n+1} - \phi^n \|_{H^k}$ becomes available, and this bound turns out to be independent on the convergence analysis. More importantly, a difference analysis between $\phi^{n+1} - \phi^n$ and $\phi^n - \phi^{n-1}$ results in a uniform in time $O (\dt^2)$ estimate for $\| \phi^{n+1} - 2 \phi^n + \phi^{n-1} \|_{H^k}$, with the constant dependent on $\frac{1}{\varepsilon}$ in a polynomial form, independent on $T$. Due to the extended structure of the BDF-style temporal stencil, in comparison with the Crank-Nicolson version~\cite{guo21}, certain recursive analysis has to be applied in the analysis for the BDF2 temporal stencil. 

  In fact, these preliminary estimates play a crucial role in the refined error analysis. The numerical error evolutionary equation is evaluated at the time instant $t^{n+1}$, and its difference with the exact PDE solution is of order $O (\dt^2 + h^4)$, with the help of $\| \phi^{n+1} - 2 \phi^n + \phi^{n-1} \|_{H^k}$ estimate, as well as the fourth order difference approximation analysis. In addition, to analyze the finite difference scheme over a uniform numerical grid, we estimate the difference between the discrete $\nabla_{h, (4)}$ norm of the numerical error function and the continuous $H^1$ norm of its continuous version, with the help of the discrete and continuous Fourier analysis. This Fourier analysis also enables us to perform an estimate for the difference between the discrete $\ell^2$ inner product associated with the nonlinear term and its continuous version, and some aliasing error control techniques have to be applied~\cite{gottlieb12b}. As a result of these preliminary estimates, we are able to apply the spectrum analysis for the linearized Cahn-Hilliard operator~\cite{alikakos94, alikakos93, chenx94}, so that all the numerical error inner product terms, in the discrete $H^{-1}$ space, are analyzed in a unified way. Such a unified error estimate leads to a coefficient of the numerical error, in the discrete $H^{-1}$ norm, independent on $\varepsilon$. Consequently, we arrive at a discrete $H^{-1}$ error estimate of order $O (\dt^2 + h^4)$, with the convergence constant in the form of $O(e^{C_0^* T} \varepsilon^{-m_0})$, with $C_0^*$ independent on $\varepsilon$.

	The outline of the paper is given as follows. In Section~\ref{sec:numerical scheme} we review the long stencil fourth order finite difference approximation, the proposed numerical scheme, and a few preliminary estimates. The main theoretical result is stated as well. The higher order $H^m$ (for $m \ge 2$) numerical stability analysis is presented in Section~\ref{sec:Hm-stab}. As a result, a uniform in time estimate for both $\| \phi^{n+1} -  \phi^n \|_{H^k}$ and $\| \phi^{n+1} - 2 \phi^n + \phi^{n-1} \|_{H^k}$ is derived as well. Subsequently, the primary convergence analysis is presented in Section~\ref{sec:convergence}. In more detail, a discrete $H^{-1}$ error estimate of order $O (\dt^2 + h^4)$ is obtained, with the convergence constant dependent on $\frac{1}{\varepsilon}$ in a polynomial form. Moreover, a three-dimensional numerical example is presented in Section~\ref{sec:numerical results}, to perform an accuracy check to validate the theoretical analysis. Some concluding remarks are made in Section~\ref{sec:conclusion}.  In the appendices we give the detailed proof of a few technical lemmas associated with the fourth order long stencil difference approximation, and provide a detailed analysis to estimate the difference between the discrete $\ell^2$ inner product and its continuous version. 	

	\section{The numerical scheme and the main convergence result} \label{sec:numerical scheme} 
	
\subsection{The spatial discretization and the related notations}

For simplicity of presentation, it is assumed that $\Omega = (0,L)^3$, and we denote $L = N \cdot h$, in which $h = \frac{L}{N}$ ($N$ being a positive integer) stands for the mesh size. In turn, the following uniform, infinite grid with grid spacing $h>0$: $E := \{ x_{i} \ |\ i\in {\mathbb{Z}}\}$, with $x_i := (i-\frac12) h$, is introduced.
	
   The derivation of long stencil fourth order finite difference formula is based on the Taylor expansion for the test function~\cite{Fornberg1988, Fornberg1998, Iserles1996, BO1978}. For example, over a one-dimensional (1-D) uniform numerical grid, the fourth order approximations to the first and second order derivatives are formulated as  
	\begin{align} 
{\cal D}_{x,(4)}^1 f_i =& \tilde{D}_x ( 1 - \frac{h^2}{6} D_x^2 ) f_i 
  = \frac{ f_{i-2} - 8 f_{i-1} + 8 f_{i+1} - f_{i+2} }{12 h} 
  = f' (x_i) + O (h^4) ,
	\label{FD-4th-1-1} 
	\\
{\cal D}_{x,(4)}^2 f_i =& D_x^2 ( 1 - \frac{h^2}{12} D_x^2 ) f_i  = \frac{ - f_{i-2} + 16 f_{i-1} - 30 f_i + 16 f_{i+1} - f_{i+2} }{12 h^2 } 
= f'' (x_i) + O (h^4) ,
	\label{FD-4th-1-2} 
	\end{align} 
in which $\tilde{D}_x$ and $D_x^2$ stand for the standard centered difference approximation to the first and second order derivatives, respectively. The long stencil fourth order difference method has been extensively applied to various physical models, such as incompressible Boussinesq equation \cite{LWJ2003, WLJ2004}, geophysical fluid models \cite{LW2008, STWW2007}, the Maxwell equation \cite{FWWY2008}, harmonic mapping flow~\cite{Xia2022a}, etc. Its advantage over the compact fourth order difference approximation~\cite{LeeC14, LiY16, Song15} has also been described in~\cite{cheng2019a}. 

 Next, the following 3-D discrete periodic function space is taken into consideration: 
	\begin{eqnarray*}
{\mathcal V}_{\rm per} &:=& \{ f: E\times E \times E \rightarrow {\mathbb{R}}\  | 
 f_{i,j,k}= \nu_{i+ \ell N,j+ m N, k + n N}, \ \forall \, i,j,k, \ell, m, n \in \mathbb{Z}  \} . 
	\end{eqnarray*}		
Meanwhile, the mean zero space is introduced to facilitate the later analysis: 
	\[
\mathring{\mathcal V}_{\rm per}:= \Big\{ f \in {\mathcal V}_{\rm per} \ \Big| \overline{f} :=  \frac{h^3}{| \Omega|} \sum_{i,j,k=1}^N f_{i,j,k}  = 0 \Big\} .
	\]
The fourth order 3-D discrete Laplacian, $\Delta_{h,(4)} : {\mathcal V}_{\rm per}\rightarrow{\mathcal V}_{\rm per}$, becomes  
	\[
\Delta_{h,(4)} := {\cal D}_{x,(4)}^2 + {\cal D}_{y, (4)}^2 + {\cal D}_{z, (4)}^2 . 
	\]
Subsequently, the discrete inner product is defined as follows:   
	\begin{eqnarray*}
\langle f , g \rangle :=  h^3\sum_{i,j,k=1}^N f_{i,j,k} g_{i,j,k},\quad f,\, g \in {\mathcal V}_{\rm per} . 
	\end{eqnarray*}	
For any $f \in\mathring{\cal V}_{\rm per}$, there is a unique solution $\msfT_h[f]\in\mathring{\cal V}_{\rm per}$ such that $-\Delta_{h, (4)} \msfT_h[f] = f$, and we denote $ (-\Delta_{h,(4)} )^{-1} f = \msfT_h[f]$. In turn, a discrete analog of the $\mathring{H}^{-1}_{\rm per}$ inner product is introduced as 
	\begin{equation} 
\langle f , g \rangle_{-1, h} := \langle f, (-\Delta_{h, (4)})^{-1} g \rangle 
= \langle  (-\Delta_{h, (4)})^{-1} f , g \rangle, \quad f ,\, g \in\mathring{\cal V}_{\rm per}. 
      \label{-1 norm-1} 
	\end{equation} 
Of course, a discrete $H^{-1}$ norm could be defined as $\| f \|_{-1,h}^2 = \langle f , f \rangle_{-1,h}$, for any $f \in\mathring{\cal V}_{\rm per}$. If $f \in {\cal V}_{\rm per}$, then $\nrm{f}_2^2 := \langle f , f \rangle$; $\nrm{f}_p^p := \langle |f|^p , 1 \rangle$ ($1\le p< \infty$), and $\nrm{f}_\infty := \max_{1\le i, j, k \le N} | f_{i,j,k} |$.

In terms of the gradient inner product, the following notations are introduced: 
\begin{equation} 
\begin{aligned} 
  & 
  [ D_x f , D_x g ]_x := \frac12 h^3 \sum_{i,j,k=1}^N ( 
   ( D_x f )_{i+\frac12,j,k} ( D_x g )_{i+\frac12,j,k} 
   + (D_x f )_{i-\frac12,j,k} ( D_x g )_{i-\frac12,j,k} ) ,  
\\
  & 
  \mbox{$[ \cdot \, , \, \cdot ]_y$ and $[ \cdot \, , \, \cdot ]_z$ could be similarly defined} , 
\\
  & 
  \langle \nabla_h f , \nabla_h g \rangle 
  := [ D_x f , D_x g ]_x + [ D_y f , D_y g ]_y + [ D_z f , D_z g ]_z ,  \quad 
  \| \nabla_h f \|_2 := ( \langle \nabla_h f , \nabla_h f \rangle )^\frac12 . 
\end{aligned} 
  \label{gradient norm-1} 
\end{equation} 

The following summation by parts formulas have been reported in~\cite{cheng2019a}: 
	\begin{align} 
	  &
 \langle f , - \Delta_{h, (4)} g \rangle  = \langle - \Delta_{h, (4)} f , g \rangle = \langle \nabla_h f , \nabla_h g \rangle + \frac{h^2}{12} ( \langle D_x^2 f , D_x^2 g \rangle 
 + \langle D_y^2 f , D_y^2 g \rangle  + \langle D_z^2 f , D_z^2 g \rangle )  , \label{summation-1} 
\\
  & 
  \langle \msfT_h f , ( - \Delta_{h, (4)}) g \rangle = \langle f , g \rangle ,  \quad \forall 
  f, g \in \mathring{\cal V}_{\rm per} .   
	\label{summation-2}
	\end{align} 
As a result, the corresponding discrete gradient norm associated with the long stencil difference is given by 
\begin{eqnarray} 
  \| \nabla_{h,(4)} f \|_2^2 = \| \nabla_h f \|_2^2 + \frac{h^2}{12} ( \| D_x^2 f \|_2^2 
   + \| D_y^2 f \|_2^2 + \| D_z^2 f \|_2^2 ) . 
  \label{energy stability-2} 
\end{eqnarray}
Furthermore, the discrete $\nrm{ \, \cdot \, }_{H_h^1}$ and $\nrm{ \, \cdot \, }_{H_h^2}$ norms are needed in the later analysis: 
	\begin{eqnarray}
\nrm{ f }_{H_h^1}^2 := \nrm{ f }_2^2 + \nrm{ \nabla_h f }_2^2 , \quad 
\nrm{ f }_{H_h^2}^2 := \nrm{ f }_{H_h^1}^2  + \nrm{ \Delta_h f }_2^2 .
	\label{discrete H1 H2 norm} 
	\end{eqnarray}
	
Meanwhile, the discrete energy is defined as follows, for any $\phi \in {\cal V}_{\rm per}$:   
\begin{eqnarray} 
    E_h (\phi) = \varepsilon^{-1} ( \frac14 \| \phi \|_4^4 - \frac12 \| \phi \|_2^2 ) + \frac{\varepsilon}{2}  \| \nabla_{h,(4)} \phi \|_2^2 .  \label{discrete energy-0} 
\end{eqnarray} 	


The inequalities in the next lemma will play an important role in the energy stability and optimal rate convergence analysis; the proof has been provided in~\cite{cheng2019a}. 

\begin{lem}  \cite{cheng2019a} \label{lem: inequality} 
  We have 
\begin{equation}  
   \| \Delta_h f \|_2 \le \| \Delta_{h, (4)} f \|_2 \le \frac43 \| \Delta_h f \|_2 ,  \, \, \, 
   \| \nabla_h f \|_2 \le \| \nabla_{h, (4)} f \|_2 \le \frac{2}{\sqrt{3}} \| \nabla_h f \|_2 ,
   \quad \forall f \in {\cal V}_{\rm per} .  
   \label{inequality-0-2}    
\end{equation} 
\end{lem} 

\subsection{A few preliminary estimates} 

For simplicity of presentation, we assume $N_x = N_y = N_z =N = 2K+1$ is odd. The general case can be analyzed in a similar manner. 

For any grid function $f \in {\cal V}_{\rm per}$, its discrete Fourier transformation is given by 
	\begin{eqnarray}
f_{i,j,k} &=& \sum^{K}_{\ell,m,n=-K}
\hat{f}^N_{\ell,m,n} {\rm e}^{2 \pi i ( \ell x_{i} 
 + m y_{j} + n z_{k} )/ L }  ,
   \label{def:Fourier}
	\end{eqnarray}
where $x_{i} = (i - \frac12 ) h$, $y_{j} = ( j - \frac12) h$, $z_{k} = (k - \frac12) h$, and $\hat{f}^N_{\ell,m,n}$ are coefficients. Subsequently, we make its extension to a continuous function:
	\begin{equation}
	\label{def:extension}
f_{{\bf F}}(x,y,z) = \sum^{K}_{\ell,m,n=-K} 
 \hat{f}^N_{\ell,m,n} {\rm e}^{2 \pi i ( \ell x + m y + n z )/ L }  .
	\end{equation}
	
  The following preliminary estimates will play important role in the analysis in later sections. The proof of Lemma~\ref{lemma:0} has been provided in~\cite{guo16}, while that of Lemmas~\ref{lemma:1} and \ref{lemma:2} will be given in Appendix~\ref{lemma 1-proof}, \ref{lemma 2-proof}. Also see the 1-D analysis in~\cite{qiao17}, with the centered difference approximation.  
  
  	\begin{lem}  \cite{guo16} 
	\label{lemma:0}
Suppose that $f , \, \Delta_h f \in {\cal V}_{\rm per}$.  Then we have 
	\begin{align}
\nrm{f}_{H_h^2} \le&  C \Big( \nrm{f}_{H_h^1}^{\frac23} \nrm{ \Delta_h^2 f}_2^{\frac13} + \nrm{f}_{H_h^1} \Big) ,
	\label{sobolev-1} 
        \\
\nrm{f}_{\infty} \le&  C \Big( \nrm{ f }_{H_h^1}^{\frac56} \nrm{ \Delta_h^2 f }_2^{\frac16} + \nrm{f}_{H_h^1} \Big) ,  
	\label{discrete-gagl-niren-1}
	\\
\nrm{\nabla_h f}_{\infty} \le&  C \nrm{ f}_{H_h^1}^{\frac12} \nrm{ \Delta_h^2 f }_2^{\frac12}  ,
	\label{discrete-gagl-niren-2} 
	\end{align}
where $C>0$ is a costant independent of $h$, and $\bar{\phi}:= \frac{1}{|\Omega|} \langle\phi,1 \rangle$ is the discrete average.
	\end{lem}

	\begin{lem}
	\label{lemma:1}
Suppose that $f , \Dh^j f \in {\cal V}_{\rm per}$. Then we have 
	\begin{align} 
  D_{2j} \| \Delta^j f_{\bf F} \| 
   \le& \| \Delta_h^j f \|_2 \le \| \Delta^j f_{\bf F} \| ,  
    \quad \forall 0  \le j \le k ,   \label{lem-1-1} 
\\ 
   D_{2j +1} \| \nabla \Delta^j f_{\bf F} \| 
   \le& \| \nabla_h \Delta_h^j f \|_2 
   \le \| \nabla \Delta^j f_{\bf F} \| ,  \quad 
   \forall 0  \le j \le k ,   
   \label{lem-1-2}  
\\
    \nrm{ f_{\bf F} }_{H^{-1}} 
   \le& \nrm{ f }_{-1,h } 
   \le D_{-1} \nrm{ f_{\bf F} }_{H^{-1} } ,   \quad 
   \mbox{if $\overline{f} =0$} , 
    \label{lem-1-3}  
\\
  \| \Delta_{h, (4)}  f \|_2  
  \le& C_1 \| \Delta_{h, (4)}^2 f \|_2  , 
	\label{H2-est-6-7} 
\\  
   \nrm{ ( \Delta_{h, (4)} f)_{\bf F} }_{H^m} 
   \le& \nrm{ \Delta f_{\bf F} }_{H^m} ,  \quad 
   \forall 0 \le m \le 2k-1 ,   
   \label{lem-1-4}        
	\end{align}
where $D_{2j}, D_{2j+1}, D_{-1}>0$, $0 \le j \le k$, and $C_1 > 0$, are constants independent of $h$.   
	\end{lem}
	
	\begin{lem}
	\label{lemma:2}
Suppose that $f \in {\cal V}_{\rm per}$. Then we have 
	\begin{equation} 
  0 \le \nrm{ \nabla f_{\bf F} }^2 - \| \nabla_{h, (4)} f \|_2^2   
   \le C h^4  \nrm{ f_{\bf F} }_{H^3}^2 ,  \label{lem-2-1}   
	\end{equation}
where $C$ is a constant independent of $h$.  
	\end{lem}

Moreover, to analyze the finite difference scheme over a uniform grid, we have to estimate the discrete nonlinear inner product. To achieve this, some tools in the Fourier pseudo-spectral analysis have to be used. Denote ${\cal B}^K$ as the space of trigonometric polynomials in $x$, $y$, and $z$ of degree up to $K$ (note that $N = 2K+1$). For a continuous function ${\mathbf f}$ with a Fourier series ${\mathbf f} (x,y,z) =  \sum_{l,m,n=-\infty}^{\infty}  \hat{\mathbf f}_{l,m,n} {\rm e}^{2 \pi {\rm i}  (l x + m y + n z) / L }$ , its projection onto the space ${\cal B}^K$ is the following truncated series  
\begin{eqnarray} 
  {\cal P}_N {\mathbf f} (x,y,z) =  \sum_{l,m,n=-K}^{K}  
 \hat{\mathbf f}_{l,m,n} {\rm e}^{2 \pi {\rm i} 
  (l x + m y + n z) / L } .   
\end{eqnarray} 
On the other hand, for ${\mathbf f}$ which may not be in ${\cal B}^K$, an interpolation operator is defined as 
\begin{equation} 
  {\cal I}_N {\mathbf f} (x,y,z) = f_{\bf F} (x,y,z), 
  \label{interpolation-1} 
\end{equation} 
in which $f$ is the discrete grid function created by the interpolation of ${\mathbf f}$: $f_{i,j,k} = {\mathbf f} (x_i, y_j, z_k)$, $0 \le i,j,k \le N-1$. Clearly ${\cal I}_N {\mathbf f} \ne {\mathbf f}$, due to the appearance of an aliasing error; their Fourier coefficients are different, unless ${\mathbf f} \in {\cal B}^K$. See the related references~\cite{Boyd2001, GO1977, HGG2007, tadmor86}, etc. On the other hand, a standard approximation analysis shows that, as long as ${\mathbf f}$ and all  derivatives (up to $m$-th order) are continuous and periodic on $\Omega$, the convergence of the derivatives of the projection and interpolation is given by 
\begin{eqnarray} 
  && 
  \| {\mathbf f} (x,y,z) 
  - {\cal P}_N {\mathbf f} (x,y,z) \|_{H^k}  
  \le  C h^{m-k} \| {\mathbf f} \|_{H^m} ,  \quad 
   0 \le k \le m ,   
   \label{spectral-approximation-1}  
\\
  &&
  \| {\mathbf f} (x,y,z) 
  - {\cal I}_N {\mathbf f} (x,y,z) \|_{H^k}  
  \le  C h^{m-k} \| {\mathbf f} \|_{H^m} , \quad 
   0 \le k \le m ,  \, m > \frac{d}{2} .  
   \label{spectral-approximation-2} 
\end{eqnarray} 
See the related discussion of approximation theory \cite{canuto82}; a similar aliasing error control result is also available in a more recent work~\cite{gottlieb12b}. 

We denote $\Phi ({\bf x},t)$ as the exact solution of the CH equation~\eqref{CH equation}, with a smooth initial data. In turn, the Fourier projection is taken as $\Phi_N ({\bf x},t) = {\cal P}_N \Phi ({\bf x},t)$. Of course, the following projection approximation is valid, for any $0 \le k \le m$: 
\begin{align} 
  & 
   \nrm{\Phi_N }_{L^\infty(0,T;H^k)}  
   \le \nrm{\Phi }_{L^\infty(0,T;H^k)} ,  \quad 
   \nrm{\Phi_N - \Phi}_{L^\infty(0,T;H^k)}  
   \le C h^{m-k} \nrm{\Phi }_{L^\infty(0,T;H^m)} , 
   \label{projection-est-1}  
\\
  &
  \| \partial_t^\ell \Phi_N \|_{L^\infty(0,T;H^k)}  
   \le \nrm{\partial_t^\ell \Phi }_{L^\infty(0,T;H^k)} ,  
   \quad  \forall \ell \ge 1 , \label{projection-est-3} 
\\
  &
  \| \partial_t^\ell ( \Phi_N - \Phi ) \|_{L^\infty(0,T;H^k)} 
   \le C h^{m-k} \| \partial_t^\ell \Phi \|_{L^\infty(0,T;H^m)},  
   \quad  \forall \ell \ge 1 . \label{projection-est-4} 
\end{align} 

The following result plays a very important role in the nonlinear inner product analysis; its proof will be given in Appendix~\ref{lemma 3-proof}. The proof in the 1-D version has been provided in~\cite{qiao17}.

	\begin{lem}
	\label{lemma:3}
Suppose $ {\mathbf f}$ and ${\mathbf g}$ are continuous and periodic functions, and their discrete interpolation grid versions are given by $f$, $g$, respectively. 

(1) If $ {\mathbf f} , {\mathbf g} \in {\cal B}^K$, we have 
\begin{equation} 
  \left\langle f , g \right\rangle 
  = \left( {\mathbf f} , {\mathbf g}  \right) . 
  \label{lem-3-1} 
\end{equation} 

(2) In general, the following estimate is valid: 
\begin{equation} 
  \left| \left\langle f , g \right\rangle 
  - \left( {\mathbf f} , {\mathbf g}  \right)  \right| 
  \le C h^8 \left( 
  \nrm{ {\mathbf f} }_{H^8}  \cdot \nrm{ {\mathbf g} }_{H^2}
   + \nrm{ {\mathbf f} }_{H^2}  
   \cdot \nrm{ {\mathbf g} }_{H^8}  \right) ,   
  \label{lem-3-2} 
\end{equation} 
in which $C$ is a constant independent of $h$.  
	\end{lem}

\subsection{The fully discrete numerical scheme and the energy stability} 

A modified BDF2 temporal discretization is applied to the Cahn-Hilliard equation, combined with long stencil fourth order difference spatial approximation in space:  
 \begin{align} 
 \frac{\frac32 \phi^{n+1}- 2 \phi^n + \frac12 \phi^{n-1}}{\dt} =& \Delta_{h,(4)}  \mu^{n+1} , \nonumber 
 \\
   \mu^{n+1} = &  
   \varepsilon^{-1} ( (\phi^{n+1})^3 - 2 \phi^n + \phi^{n-1} )    
   - \varepsilon \Delta_{h,(4)}\phi^{n+1}   \nonumber 
\\
  & 
   - A \varepsilon^{-2} \dt \Delta_{h,(4)} ( \phi^{n+1} - \phi^n ) .  
 \label{scheme-BDF-CH-1}
 \end{align}


In terms of the convergence analysis, we compare the numerical solution with the projection solution $\Phi_N$, instead of the exact solution $\Phi$, to simplify the discrete $H^{-1}$ error estimate. By $\Phi_N^m$ we denote $\Phi_N(\, \cdot \, , t^m)$, with $t^m = m\cdot \dt$. Since $\Phi_N \in {\cal B}^K$, the mass conservative property is available at the discrete level: 
	\begin{equation} 
\overline{\Phi_N^m} = \frac{1}{|\Omega|}\int_\Omega \, \Phi_N ( \cdot, t_m) \, d {\bf x} = \frac{1}{|\Omega|}\int_\Omega \, \Phi_N ( \cdot, t_{m-1}) \, d {\bf x} = \overline{\Phi_N^{m-1}} ,  \quad \forall \ m \in\mathbb{N}.  
	\label{mass conserv-1} 
	\end{equation} 
On the other hand, the solution of the proposed scheme~\eqref{scheme-BDF-CH-1} is also mass conservative at the discrete level: 
	\begin{equation} 
\overline{\phi^m} = \overline{\phi^{m-1}} ,  \quad \forall \ m \in \mathbb{N} .  
	\label{mass conserv-2} 
	\end{equation} 
Meanwhile, we use the mass conservative projection for the initial data:  
	\begin{equation}
\phi^0_{i,j,k} = \Phi_N (x_i, y_j, z_k, t=0) . 
	\label{initial data-0}
	\end{equation}	
The error grid function is defined as 
	\begin{equation} 
e^k := \Phi_N^k - \phi^k ,  \quad 
 \forall \ k \in\mathbb{N} .   
	\label{error function-1}
	\end{equation} 
Therefore, it follows that  $\overline{e^k} =0$, for any $k \in \mathbb{N}$, so that the discrete norm $\nrm{ \, \cdot \, }_{-1, h}$ is well defined for the error grid function. 

Since the proposed numerical method~\eqref{scheme-BDF-CH-1} is a two-step algorithm, a ``ghost" point extrapolation for $\phi^{-1}$ is needed to preserve the second order accuracy in time. As outlined above, the initial value is taken as $\phi^0_{i,j,k} = \Phi_N (x_i, y_j, z_k, t=0)$. At the ``ghost" time instant $t^{-1}$, the following approximation is taken: 
\begin{equation} 
  \phi^{-1} = \phi^0 - \dt \Delta_{h,(4)} \mu_h^0, \quad \mbox{with} \, \, \mu_h^0 := \varepsilon^{-1} ( (\phi^0)^3 - \phi^0 ) - \varepsilon \Delta_{h,(4)}\phi^0 .  \label{scheme-BDF-CH-initial-1}
\end{equation}
Of course, the Taylor expansion implies an $O (\dt^2 + h^4)$ accuracy for this approximation: 
\begin{eqnarray} 
  \| \phi^{-1} - \Phi^{-1} \|_2 \le C (\dt^2 + h^4). \label{scheme-BDF-CH-initial-2}
\end{eqnarray}  

The unique solvability, energy stability, as well as a discrete $H^1$ stability of the numerical solution, have been proved in~\cite{cheng2019a}. 

\begin{thm} \cite{cheng2019a} \label{CH solvability} 
  Given $\phi^k, \phi^{k-1} \in {\mathcal V}_{\rm per}$, with $\overline{\phi^k} = \overline{\phi^{k-1}}$, there exists a unique solution $\phi^{k+1} \in {\mathcal V}_{\rm per}$ for the numerical scheme~\eqref{scheme-BDF-CH-1}. And also, this scheme is mass conservative, i.e., $\overline{\phi^k} \equiv \overline{\phi^0} := \beta_0$, for any $k \ge 0$. For $k \ge 1$, we introduce
	\begin{equation}
\mathcal{E}_h (\phi^{k+1}, \phi^k) := E_h ( \phi^{k+1})+\frac{1}{4\dt} \| \phi^{k+1}- \phi^k\|_{-1,h}^2 + \frac{1}{2 \varepsilon} \| \phi^{k+1} - \phi^k \|_2^2 . 
	\label{discrete energy}
	\end{equation}
For $A \ge \frac{1}{16}$, a modified energy-decay property is available for the numerical scheme~\eqref{scheme-BDF-CH-1}: 
	\begin{equation}
\mathcal{E}_h ( \phi^{k+1}, \phi^k) 
\le \mathcal{E}_h ( \phi^k, \phi^{k-1}) .
	\label{CH-eng stab-est}
	\end{equation}
Furthermore, suppose that the initial data are sufficiently regular so that
	\[
E_h (\phi^0) + \frac{\dt}{4} \| \nabla_{h, (4)} \mu_h^0 \|_2^2 + \frac{\dt^2}{2} \| \Delta_{h, (4)} \mu_h^0 \|_2^2 \le \tilde{C}_0,
	\]
for some $\tilde{C}_0$ that is independent of $h$, and $A\ge \frac{1}{16}$. Then we have the following uniform (in time) $H_h^1$ bound for the numerical solution: 
	\begin{equation}
\nrm{ \phi^m}_{H_h^1} \le \tilde{C}_{1, \varepsilon} = O (\varepsilon^{-1}) ,  \quad \forall m \ge 1 , \label{CH-H1 stab-0}
	\end{equation} 
in which $\tilde{C}_{1, \varepsilon}$ only depends on $\Omega$, $\varepsilon$ and $\tilde{C}_0$, independent on $h$, $\dt$ and final time. In more detail, its dependence on $\varepsilon^{-1}$ is in a polynomial form, namely $O (\varepsilon^{-1})$.   
	\end{thm}
	
As a consequence of Theorem~\ref{CH solvability} and the estimate (\ref{lem-1-2}), combined with the definition $\nrm{ \phi_{\bf F} }_{H^1}^2 = \nrm{ \phi_{\bf F} }^2 + \nrm{ \nabla \phi_{\bf F} }^2$, the following result is obvious. 
	
	\begin{cor}  \label{stab-H1-2}
We have 
	\begin{align}
\nrm{\phi_{\bf F} }_{\ell^\infty(0,T;H^1)} :=& 
\max_{0\le m\le M}\nrm{\phi^m_{\bf F} }_{H^1} 
 \le \hat{C}_{1,\varepsilon} 
  := C \varepsilon^{-1}  ,
	\label{LinftyH1-stability-bound-2}
	\end{align}
in which $C$ is a constant independent of $h$, $\dt$, $T$, and $\varepsilon$. 
	\end{cor}
	
\subsection{Statement of the main theorem} 

The error estimate for the numerical scheme~\eqref{scheme-BDF-CH-1} has been reported in~\cite{cheng2019a}, with second order accuracy in time and fourth order accuracy in space. However, a standard error estimate using the discrete Gronwall inequality has always contained a convergence constant singularly dependent on $\varepsilon^{-1}$ in an exponential form. 

A refined error estimate is provided in this article. The following theorem is the main theoretical result. 

	\begin{thm}
	\label{thm:convergence}
Suppose that the initial data $\phi_0$ is sufficiently smooth and that $\dt$ and $h$ satisfy the scaling laws 
	\begin{equation}
  \dt  \le \hat{C} \varepsilon^{J_1} ,  \quad h \le \hat{C} \varepsilon^{J_2} ,
	\label{convergence-condition-1}
	\end{equation}
where $J_1$ and $J_2$ are positive integers that are sufficiently large, and $\hat{C}$ is a constant for a fixed final time $T$. Also assume that $0 < \varepsilon < \varepsilon_0$, with $\varepsilon_0$ specified in Proposition \ref{Feng04NM}. Then the following convergence result is valid:
	\begin{equation}
 \max_{1 \le m \le M} \| e^m \|_{-1, h} 
  \le \hat{R}^* ( \dt^2 + h^4 ) , \quad  
   \mbox{with} \, \, M = \left[ \frac{T}{\dt} \right] , \, \, 
   \hat{R}^* = C e^{C_0^* T} \varepsilon^{- J_0} , 
	\label{convergence-1}
	\end{equation}
where $J_0$ is a positive integer, $C_0^*$ and $C$ are positive constants that are independent of $\dt$, $h$ and $\varepsilon$. 
	\end{thm}

	\section{Higher order Sobolev estimates of the numerical scheme}  
	\label{sec:Hm-stab}
	
\subsection{$\ell^\infty \left( 0, T ; H_h^m \right)$ ($m \ge 2$) bound of the scheme}

A uniform in time $H_h^1$ bound of the numerical solution has been proved in Theorem~\ref{CH solvability}, as presented in~\cite{cheng2019a}. However, such a bound is not sufficient for the desired estimate of a refined convergence constant. In this section, we establish a uniform in time $H_h^m$ bound, for any $m \ge 2$, of the numerical solution. More importantly, the derived bound depends on $\varepsilon^{-1}$ in a polynomial form.  

	\begin{thm} \label{stab-H2-1}  
For the proposed numerical scheme~\eqref{scheme-BDF-CH-1}, if $\dt \le \frac{\varepsilon}{2 \sqrt{2} A^\frac12}$, we have 
	\begin{align}
\nrm{\phi_{\bf F} }_{\ell^\infty(0,T;H^2)} :=& 
\max_{0\le m\le M}\nrm{\phi^m_{\bf F} }_{H^2} 
 \le \hat{C}_{2,\varepsilon} 
  := C \varepsilon^{- k_2} ,
	\label{LinftyH2-stability-bound}
	\end{align}
in which $k_2$ is a positive integer and $C >0$ is a constant independent of $h$, $\dt$, $T$, and $\varepsilon$. 
	\end{thm}
	
	\begin{proof}
Taking the discrete inner product with~\eqref{scheme-BDF-CH-1} by $2 \Delta_{h, (4)}^2 \phi^{n+1}$ gives  
	\begin{align} 
& \langle 3 \phi^{n+1} - 4 \phi^n + \phi^{n-1} , \Delta_{h, (4)}^2  \phi^{n+1} \rangle 
 + 2 \varepsilon \dt \| \Delta_{h, (4)}^2 \phi^{n+1} \|_2^2   \nonumber 
\\
  & 
 + 2 A \varepsilon^{-2} \dt^2 \langle \Delta_{h, (4)}^2 ( \phi^{n+1} - \phi^n ) , 
   \Delta_{h, (4)}^2 \phi^{n+1}  \rangle 
	\nonumber 
	\\
&  = 
   - 2 \varepsilon^{-1} \dt \langle \Delta_{h, (4)} ( 2 \phi^n - \phi^{n-1} ) , 
  \Delta_{h, (4)}^2 \phi^{n+1} \rangle 
 + 2 \varepsilon^{-1}\dt \langle \Delta_{h, (4)}  ( \phi^{n+1} )^3 ,  
  \Delta_{h, (4)}^2 \phi^{n+1} \rangle .   
	\label{H2-est-1}
	\end{align}
The temporal differentiation term could be analyzed as follows, with the help of summation-by-parts formula:  
\begin{align} 
  &
   \langle 3 \phi^{n+1} - 4 \phi^n + \phi^{n-1} , \Delta_{h, (4)}^2 \phi^{n+1} \rangle  
   = \langle  \Delta_{h, (4)} ( 3 \phi^{n+1} - 4 \phi^n + \phi^{n-1} ) , 
    \Delta_{h, (4)} \phi^{n+1} \rangle   
   \nonumber 
\\
  =&
     \frac12 ( \|  \Delta_{h, (4)} \phi^{n+1} \|_2^2 - \| \Delta_{h, (4)} \phi^n \|_2^2 
   + \| \Delta_{h, (4)} ( 2 \phi^{n+1} - \phi^n ) \|_2^2 
   - \| \Delta_{h, (4)} ( 2 \phi^n - \phi^{n-1} ) \|_2^2   \nonumber 
\\
  &
    +  \| \Delta_{h, (4)} ( \phi^{n+1} - 2 \phi^n + \phi^{n-1} ) \|_2^2 ) . \label{H2-est-2} 
\end{align} 
A triangular equality could be applied to the artificial regularization term:  
	\begin{align}
2 \langle \Delta_{h, (4)}^2 ( \phi^{n+1} - \phi^n ) , \Delta_{h, (4)}^2 \phi^{n+1}  \rangle   
=&  \| \Delta_{h, (4)}^2 \phi^{n+1} \|_2^2 - \| \Delta_{h, (4)}^2 \phi^n \|_2^2  
 +    \| \Delta_{h, (4)}^2 ( \phi^{n+1} - \phi^n)  \|_2^2   .    
	\label{H2-est-3}
	\end{align}
In terms of the concave diffusion term, we denote $\hat{\phi}^{n+1} = 2 \phi^n - \phi^{n-1}$ and see that 
	\begin{align}  
	  &
- \langle  \Delta_{h, (4)}  \hat{\phi}^{n+1} ,  \Delta_{h, (4)}^2 \phi^{n+1} \rangle \le   \alpha \| \Delta_{h, (4)}^2 \phi^{n+1} \|_2^2 + \frac{1}{4 \alpha}  \| \Delta_{h, (4)} \hat{\phi}^{n+1}  \|_2^2 
	\nonumber
	\\
	\le  & 
	\alpha \| \Delta_{h, (4)}^2 \phi^{n+1} \|_2^2 + \frac{4}{9 \alpha}  \| \Delta_h \hat{\phi}^{n+1}  \|_2^2  \le \alpha \| \Delta_{h, (4)}^2 \phi^{n+1} \|_2^2 
	 + \frac{8}{3 \alpha}  \| \Delta_h \phi^n  \|_2^2 
	 + \frac{4}{3 \alpha}  \| \Delta_h \phi^{n-1}  \|_2^2 ,  
    \label{H2-est-4-1}  
	\end{align}
for any $\alpha >0$, in which inequality~\eqref{inequality-0-2} has been applied at the second step. Meanwhile, the quantities $\| \Delta_h  \phi^n \|_2^2$, $\| \Delta_h  \phi^{n-1} \|_2^2$ can be controlled by
	\begin{equation} 
\begin{aligned} 
\| \Delta_h  \phi^\ell \|_2^2 \le & \| \nabla_h \phi^\ell \|_2^\frac43 \cdot \| \Delta_h^2 \phi^\ell \|_2^\frac23  \le \frac{2}{3 \alpha} \| \nabla_h \phi^\ell \|_2^2 + \frac{\alpha^2}{3} \| \Delta_h^2 \phi^\ell \|_2^2  
\\
  \le &  \frac{2}{3 \alpha} \| \nabla_h \phi^\ell \|_2^2 + \frac{\alpha^2}{3} 
  \| \Delta_{h, (4)}^2 \phi^\ell \|_2^2
   \le  \frac{2 \tilde{C}_{1, \varepsilon}^2}{3 \alpha} + \frac{\alpha^2}{3} 
  \| \Delta_{h, (4)}^2 \phi^\ell \|_2^2  ,  \quad \ell = n, n-1 , \, \,  \forall \alpha > 0 , 
\end{aligned} 
	\label{H2-est-4-2}    
	\end{equation}
in which the first step corresponds to a Sobolev interpolation inequality, the second step comes from an application of Young's inequality, the third step is based on inequality~\eqref{inequality-0-2}, and the preliminary $H_h^1$ estimate~\eqref{CH-H1 stab-0} is recalled in the last step. In turn, a combination of (\ref{H2-est-4-1}) and (\ref{H2-est-4-2}) leads to 
	\begin{align} 
  - \langle  \Delta_{h, (4)}  \hat{\phi}^{n+1} ,  \Delta_{h, (4)}^2 \phi^{n+1} \rangle 
   \le  \alpha \| \Delta_{h, (4)}^2 \phi^{n+1} \|_2^2   + \frac{8 \alpha}{9}  \| \Delta_{h, (4)}^2 \phi^n \|_2^2 + \frac{4 \alpha}{9}  \| \Delta_{h, (4)}^2 \phi^{n-1} \|_2^2 + \frac{8 \tilde{C}_{1, \varepsilon}^2}{3 \alpha^2}  , 
	\label{H2-est-4-3}
	\end{align}
for any $\alpha > 0$. 

Regarding the nonlinear term, we begin with the following observation:  
	\begin{equation} 
\begin{aligned} 
	\langle \Delta_{h, (4)}  ( \phi^{n+1} )^3 ,  \Delta_{h, (4)}^2 \phi^{n+1} \rangle  
	\le & \| \Delta_{h, (4)}  ( \phi^{n+1} )^3 \|_2 \cdot   
   \| \Delta_{h, (4)}^2 \phi^{n+1} \|_2  
      \le  \frac43  \| \Delta_h  ( \phi^{n+1} )^3 \|_2 \cdot   
   \| \Delta_{h, (4)}^2 \phi^{n+1} \|_2 ,  
\end{aligned} 
	\label{H2-est-5-1}    
	\end{equation} 
in which inequality~\eqref{inequality-0-2} has been applied again. The rest work is focused on obtaining a bound for $ \| \Delta_h  ( \phi^{n+1} )^3 \|_2$. Detailed expansions and repeated applications of the discrete H\"older inequality yields the following estimate (also see the derivation in~\cite{guo16}): 
	\begin{align}
\nrm{ \Delta_h ( f g h ) }_2 \le&  C  \Bigl(  \nrm{ f }_\infty \cdot \nrm{ g }_\infty  \cdot \nrm{ h }_{H_h^2}  +  \nrm{ f }_\infty \cdot \nrm{ h }_\infty  \cdot \nrm{ g }_{H_h^2}  
	\nonumber 
	\\
& +  \nrm{ g }_\infty \cdot \nrm{ h }_\infty  \cdot \nrm{ f }_{H_h^2}  + \nrm{ f }_\infty \cdot \nrm{ \nabla_h g }_\infty  \cdot \nrm{ \nabla_h h }_2 
	\nonumber 
	\\
& + \nrm{ g }_\infty \cdot \nrm{ \nabla_h f }_\infty  \cdot \nrm{ \nabla_h h }_2 + \nrm{ h }_\infty \cdot \nrm{ \nabla_h f }_\infty  \cdot \nrm{ \nabla_h g }_2 \Bigr)  .
	\label{H2-est-5-3}
	\end{align}
In turn, a substitution of $f=g=h=\phi^{n+1}$ gives  
	\begin{equation}  
	\| \Delta_h  ( \phi^{n+1} )^3 \|_2 
	\le   C  \nrm{ \phi^{n+1} }_\infty  \cdot  \nrm{ \nabla_h \phi^{n+1} }_\infty 
	\cdot  \nrm{ \phi^{n+1} }_{H_h^1} 
   + C   \nrm{ \phi^{n+1} }_\infty^2 \cdot   \nrm{ \phi^{n+1} }_{H_h^2} .
	\label{H2-est-5-4}
	\end{equation}
The uniform in time $H_h^1$ estimate~\eqref{CH-H1 stab-0}, combined with the 3-D discrete Sobolev inequality~\eqref{sobolev-1}, and the discrete Gagliardo-Nirenberg type inequalities~\eqref{discrete-gagl-niren-1} and  \eqref{discrete-gagl-niren-2}, yields
	\begin{align}
 \| \phi^\ell \|_{H_h^2} \le&  C ( \| \phi^\ell \|_{H_h^1}^{\frac23} \| \Delta_h^2 \phi^\ell \|_2^{\frac13} + \| \phi^\ell \|_{H_h^1} )   \le C \Big( \tilde{C}_{1, \varepsilon}^{\frac23} \cdot \| \Delta_h^2 \phi^\ell \|_2^{\frac13} +  \tilde{C}_{1, \varepsilon} \Big) ,   \label{H2-est-5-5-2} 
            \\
  \| \phi^\ell \|_{\infty} \le&  C ( \| \phi^\ell \|_{H_h^1}^{\frac56} \| \Delta_h^2 \phi^\ell \|_2^{\frac16} + \| \phi^\ell \|_{H_h^1} )   \le  C \Big( \tilde{C}_{1, \varepsilon}^{\frac56} \cdot \| \Delta_h^2 \phi^\ell \|_2^{\frac16} + \tilde{C}_{1, \varepsilon} \Big) ,  
	\label{H2-est-5-5-3}
            \\
 \| \nabla_h \phi^\ell \|_{\infty} \le&  C_{11} \| \phi^\ell \|_{H_h^1}^{\frac12} \cdot \| \Delta_h^2 \phi^\ell \|_2^{\frac12}  \le C \tilde{C}_{1, \varepsilon}^{\frac12} \cdot \| \Delta_h^2 \phi^\ell \|_2^{\frac12}  ,  
	\label{H2-est-5-5-4}
	\end{align}
for $\ell= n+1$. Going back~\eqref{H2-est-5-4}, we obtain  
	\begin{equation} 
   \| \Delta_h  ( \phi^{n+1} )^3 \|_2  
   \le C \Big( \tilde{C}_{1, \varepsilon}^{\frac73} \cdot \| \Delta_h^2 \phi^{n+1} \|_2^{\frac23} 
     +  \tilde{C}_{1, \varepsilon}^3 \Big)  \le  C \Big( \tilde{C}_{1, \varepsilon}^{\frac73} \cdot \| \Delta_{h, (4)}^2 \phi^{n+1} \|_2^{\frac23} 
     +  \tilde{C}_{1, \varepsilon}^3 \Big)  ,
	\label{H2-est-5-6}
	\end{equation}
with an application of inequality~\eqref{inequality-0-2} in the second step. 
A combination of~\eqref{H2-est-5-6} and \eqref{H2-est-5-1} implies that 
	\begin{equation} 
\begin{aligned} 
	\langle \Delta_{h, (4)}  ( \phi^{n+1} )^3 ,  \Delta_{h, (4)}^2 \phi^{n+1} \rangle  
	\le & C \Big( \tilde{C}_{1, \varepsilon}^{\frac73} \cdot \| \Delta_{h, (4)}^2 \phi^{n+1} \|_2^{\frac53} +  \tilde{C}_{1, \varepsilon}^3 \cdot \| \Delta_{h, (4)}^2 \phi^{n+1} \|_2 \Big) 
\\
   \le & 
      C  \alpha^{-5} \tilde{C}_{1, \varepsilon}^{14}  
      + \alpha \| \Delta_{h, (4)}^2 \phi^{n+1} \|_2^2 , \quad 
     \forall \alpha > 0 ,   
\end{aligned} 
	\label{H2-est-5-7}    
	\end{equation} 
in which the Young's inequality has been applied. 

Therefore, a substitution~\eqref{H2-est-2}, \eqref{H2-est-3}, \eqref{H2-est-4-3} and \eqref{H2-est-5-7} into \eqref{H2-est-1} results in 
	\begin{align} 
& \frac12 ( \|  \Delta_{h, (4)} \phi^{n+1} \|_2^2 - \| \Delta_{h, (4)} \phi^n \|_2^2 
   + \| \Delta_{h, (4)} ( 2 \phi^{n+1} - \phi^n ) \|_2^2 
   - \| \Delta_{h, (4)} ( 2 \phi^n - \phi^{n-1} ) \|_2^2 )   \nonumber 
\\
  &   
 + 2 A \varepsilon^{-2} \dt^2 (  \| \Delta_{h, (4)}^2 \phi^{n+1} \|_2^2 
 - \| \Delta_{h, (4)}^2 \phi^n \|_2^2  )  
 + 2 ( \varepsilon - 2 \alpha \varepsilon^{-1} ) \dt 
 \| \Delta_{h, (4)}^2 \phi^{n+1} \|_2^2  \nonumber 
	\\
 \le &  
    \frac{16 \alpha \varepsilon^{-1} \dt}{9}  \| \Delta_{h, (4)}^2 \phi^n \|_2^2 
   + \frac{8 \alpha \varepsilon^{-1} \dt}{9}  \| \Delta_{h, (4)}^2 \phi^{n-1} \|_2^2  
   +  C_{4, \varepsilon} \dt,  \label{H2-est-6-1} 
\\
  & 
    \mbox{with} \quad  
     C_{4, \varepsilon} = \frac{16 \tilde{C}_{1, \varepsilon}^2 \varepsilon^{-1} }{3 \alpha^2}  
     + C  \alpha^{-5} \tilde{C}_{1, \varepsilon}^{14} \varepsilon^{-1}  .   
	\nonumber 
	\end{align}
Subsequently, by setting $\alpha = \frac{1}{16} \varepsilon^2$, we get 
	\begin{align} 
& \frac12 ( \|  \Delta_{h, (4)} \phi^{n+1} \|_2^2  
   + \| \Delta_{h, (4)} ( 2 \phi^{n+1} - \phi^n ) \|_2^2  ) 
   + 2 A \varepsilon^{-2} \dt^2  \| \Delta_{h, (4)}^2 \phi^{n+1} \|_2^2  
   + \frac{7 \varepsilon \dt}{4} \| \Delta_{h, (4)}^2 \phi^{n+1} \|_2^2   \nonumber 
\\
  \le &   
   \frac12 ( \| \Delta_{h, (4)} \phi^n \|_2^2 
   + \| \Delta_{h, (4)} ( 2 \phi^n - \phi^{n-1} ) \|_2^2 )  
  + 2 A \varepsilon^{-2} \dt^2 \| \Delta_{h, (4)}^2 \phi^n \|_2^2   \nonumber 
	\\
    &  
    + \frac{\varepsilon \dt}{18}  ( 2 \| \Delta_{h, (4)}^2 \phi^n \|_2^2 
   +  \| \Delta_{h, (4)}^2 \phi^{n-1} \|_2^2 )   
   +  C_{4, \varepsilon} \dt .   
	\label{H2-est-6-2}
	\end{align} 
For simplicity, we denote $F^n := \frac12 ( \| \Delta_{h, (4)} \phi^n \|_2^2 
   + \| \Delta_{h, (4)} ( 2 \phi^n - \phi^{n-1} ) \|_2^2 )  
  + 2 A \varepsilon^{-2} \dt^2 \| \Delta_{h, (4)}^2 \phi^n \|_2^2$. Meanwhile, an addition of $\frac{2 \varepsilon \dt}{9} \| \Delta_{h, (4)}^2 \phi^n \|_2^2$ to both sides of the last inequality yields  
	\begin{align} 
  & F^{n+1} 
   + \frac{7 \varepsilon \dt}{4} \| \Delta_{h, (4)}^2 \phi^{n+1} \|_2^2  
   + \frac{2 \varepsilon \dt}{9} \| \Delta_{h, (4)}^2 \phi^n \|_2^2  \nonumber 
\\
  \le &   
   F^n 
    + \frac{\varepsilon \dt}{3}   \| \Delta_{h, (4)}^2 \phi^n \|_2^2 
   +  \frac{\varepsilon \dt}{18} \| \Delta_{h, (4)}^2 \phi^{n-1} \|_2^2    
   +  C_{4, \varepsilon} \dt .   
	\label{H2-est-6-3}
	\end{align} 
In turn, a modified functional quantity is introduced as 
	\begin{equation} 
\begin{aligned} 
G^n := & F^n 
    + \frac{\varepsilon \dt}{3}   \| \Delta_{h, (4)}^2 \phi^n \|_2^2 
   +  \frac{\varepsilon \dt}{18} \| \Delta_{h, (4)}^2 \phi^{n-1} \|_2^2 
\\
  = & 
  ( \frac12 +  2 A \varepsilon^{-2} \dt^2 ) \| \Delta_{h, (4)} \phi^n \|_2^2 
   + \frac12 \| \Delta_{h, (4)} ( 2 \phi^n - \phi^{n-1} ) \|_2^2   
\\
  & 
    + \frac{\varepsilon \dt}{3}   \| \Delta_{h, (4)}^2 \phi^n \|_2^2 
   +  \frac{\varepsilon \dt}{18} \| \Delta_{h, (4)}^2 \phi^{n-1} \|_2^2 ,  
\end{aligned} 
	\label{H2-est-6-5}    
	\end{equation} 
so that the following inequality is valid: 
	\begin{equation} 
G^{n+1} + \frac{17 \varepsilon \dt}{12}   \| \Delta_{h, (4)}^2 \phi^{n+1} \|_2^2 
   +  \frac{\varepsilon \dt}{6} \| \Delta_{h, (4)}^2 \phi^n \|_2^2 
   \le G^n + C_{4, \varepsilon} \dt .
	\label{H2-est-6-6}
	\end{equation}
On the other hand, the following estimates could be derived, with the help of inequality \eqref{H2-est-6-7}:  
\begin{equation} 
\begin{aligned} 
  & 
  ( \frac12 +  2 A \varepsilon^{-2} \dt^2 ) \| \Delta_{h, (4)} \phi^{n+1} \|_2^2  
  \le \frac34 \| \Delta_{h, (4)} \phi^{n+1} \|_2^2  
  \le \frac34 C_1^2  \| \Delta_{h, (4)}^2 \phi^{n+1} \|_2^2  ,   \, \, 
  \mbox{if $\dt \le \frac{\varepsilon}{2 \sqrt{2} A^\frac12}$} , 
\\
  & 
  \frac12 \| \Delta_{h, (4)} ( 2 \phi^{n+1} - \phi^n ) \|_2^2  
  \le 3 \| \Delta_{h, (4)}  \phi^{n+1} \|_2^2  
  + \frac32 \| \Delta_{h, (4)}  \phi^n \|_2^2  
\\
  &  \qquad \qquad \qquad \qquad \qquad \quad 
  \le  3 C_1^2 \| \Delta_{h, (4)}^2  \phi^{n+1} \|_2^2  
  + \frac{3 C_1^2}{2} \| \Delta_{h, (4)}^2  \phi^n \|_2^2 , 
\\
  & \mbox{so that} \, \, \, 
   G^{n+1} =
  ( \frac12 +  2 A \varepsilon^{-2} \dt^2 ) \| \Delta_{h, (4)} \phi^{n+1} \|_2^2 
   + \frac12 \| \Delta_{h, (4)} ( 2 \phi^{n+1} - \phi^n ) \|_2^2   
\\
  &  \qquad \qquad  \qquad  \quad 
    + \frac{\varepsilon \dt}{3}   \| \Delta_{h, (4)}^2 \phi^{n+1} \|_2^2 
   +  \frac{\varepsilon \dt}{18} \| \Delta_{h, (4)}^2 \phi^n \|_2^2 
\\
  &  \qquad \qquad \qquad  
   \le  4 C_1^2 \| \Delta_{h, (4)}^2  \phi^{n+1} \|_2^2  
  + 2 C_1^2 \| \Delta_{h, (4)}^2  \phi^n \|_2^2  ,   \quad \mbox{if} \, \, 
  \dt \le \frac34 C_1^2 \varepsilon^{-1} , 
\\
  & \mbox{and} \quad  
  \frac{17 \varepsilon \dt}{12}   \| \Delta_{h, (4)}^2 \phi^{n+1} \|_2^2 
   +  \frac{\varepsilon \dt}{6} \| \Delta_{h, (4)}^2 \phi^n \|_2^2  
   \ge \frac{\varepsilon \dt}{12 C_1^2} G^{n+1}  
   +   \varepsilon \dt   \| \Delta_{h, (4)}^2 \phi^{n+1} \|_2^2 . 
\end{aligned} 
  \label{H2-est-6-8}
\end{equation} 
Going back~\eqref{H2-est-6-6}, we arrive at 
	\begin{equation} 
\Big( 1 + \frac{ \varepsilon \dt}{12 C_1^2}  \Big) G^{n+1} 
+ \varepsilon \dt   \| \Delta_{h, (4)}^2 \phi^{n+1} \|_2^2 
  \le   G^n  + C_{4, \varepsilon}  \dt  .
	\label{H2-est-6-9}    
	\end{equation}
An application of recursive argument to the above inequality reveals that
	\begin{equation}
\begin{aligned} 
  & 
G^{n+1} \le 
  \Big(1+\frac{ \varepsilon^2 \dt}{12 C_1^2} \Big)^{-(n+1)} G^0 
  + 12 C_1^2  \varepsilon^{-1} C_{4, \varepsilon} 
   \le G^0 + 12 C_1^2  \varepsilon^{-1} C_{4, \varepsilon}  ,   \quad \mbox{so that} 
\\
  &
   \| \Delta_h \phi^{n+1} \|_2^2  \le  \| \Delta_{h, (4)} \phi^{n+1} \|_2^2  
   \le 2 G^{n+1} \le 2 ( G^0 + 12 C_1^2  \varepsilon^{-1} C_{4, \varepsilon} ), \quad 
   \forall n \ge 0 . 
\end{aligned} 
	\label{H2-est-6-11}
	\end{equation}
As a consequence, \eqref{LinftyH2-stability-bound} is a direct consequence of \eqref{H2-est-6-11}  and the elliptic regularity: 
\begin{align} 
  \nrm{ \phi_{\bf F}^{n+1} }_{H^2}  
  \le& C  (   
   \|  \phi_{\bf F}^{n+1} \| 
   +  \| \Delta  \phi_{\bf F}^{n+1} \|  )  
  \le C  \left(  \hat{C}_{1, \varepsilon}  
   +  \| \Delta_h  \phi^{n+1} \|_2   \right) 
   \nonumber 
\\
  \le& 
   C  (  \hat{C}_{1, \varepsilon}  +  ( 2 G^{n+1} )^\frac12   ) 
   \le C  (  \hat{C}_{1, \varepsilon}  
   +  \sqrt{2} ( G^0 )^{1/2} 
  + 2 \sqrt{6} C_1 C_{4, \varepsilon}^\frac12 \varepsilon^{-\frac12}   )  
  : = \hat{C}_{2, \varepsilon} , 
   \label{H2-est-6-12}    
	\end{align} 
in which inequality~\eqref{lem-1-1} has been applied in the second step. It is clear that $\hat{C}_{2, \varepsilon}$ depends on $\varepsilon^{-1}$ in a polynomial form, since $C_{4, \varepsilon}$ does. The proof of Theorem~\ref{stab-H2-1} is complete. 
	\end{proof}

Using numerical analysis techniques, a uniform in time $H^{m_0}$ bound for the numerical solution could be established, for any $m_0 \ge 2$. To obtain such a bound, a discrete inner product with~\eqref{scheme-BDF-CH-1} is taken by $(- \Delta_{h, (4)} )^{m_0} \phi^{n+1}$, and repeated application of discrete H\"older inequality, Sobolev embedding, as well as Gagliardo-Nirenberg inequality, will play a crucial role in the derivation. The details are left for interested readers. 

	\begin{thm} \label{stab-Hm-1}  
For the proposed numerical scheme~\eqref{scheme-BDF-CH-1}, if $\dt \le C \varepsilon$, we have 
	\begin{align}
\nrm{\phi_{\bf F} }_{\ell^\infty(0,T;H^{m_0})} :=& 
\max_{0\le m\le M}\nrm{\phi^m_{\bf F} }_{H^{m_0} } 
 \le \hat{C}_{m_0, \varepsilon} 
  := C \varepsilon^{- k_{m0} } ,
	\label{LinftyHm-stability-bound}
	\end{align}
where $k_{m0}$ is a positive integer and $C>0$ is a constant independent of $h$, $\dt$, $T$, and $\varepsilon$. 
	\end{thm}

\subsection{A uniform estimate for $\nrm{ \phi^{n+1} - \phi^n }_{H_h^k}$ } 

The following estimate is needed in the later analysis. 

	\begin{thm} \label{stab-time-1}
For the proposed numerical scheme~\eqref{scheme-BDF-CH-1}, if $\dt \le C \varepsilon$,  we have 
	\begin{eqnarray}
\max_{0\le n \le M-1} \nrm{ \phi^{n+1}_{\bf F} 
- \phi^n_{\bf F} }_{H^k} 
 \le \hat{D}_{k, \varepsilon} \dt ,  \quad 
  \mbox{with} \, \,  \hat{D}_{k, \varepsilon} 
  := C \varepsilon^{- n_k } ,
	\label{time-stability-1-bound}
	\end{eqnarray}
where $n_k$ is a positive integer and $C>0$ is a constant independent of $h$, $\dt$, $T$, and $\varepsilon$. 
	\end{thm}
	
\begin{proof}
A careful evaluation of the numerical scheme~\eqref{scheme-BDF-CH-1} indicates that 
 	\begin{eqnarray}
 \Big\| ( \frac32 \phi^{n+1}- 2 \phi^n + \frac12 \phi^{n-1} )_{\bf F} \Big\|_{H^k} = 
  \dt \| ( \Delta_{h, (4)} \mu^{n+1} )_{\bf F} \|_{H^k} 
  \le  \dt \| \Delta ( \mu^{n+1} )_{\bf F} \|_{H^k} ,  \label{est-time-1-1} 
    \end{eqnarray} 
in which the last step comes from estimate~\eqref{lem-1-4}). On the other hand, the term $( \mu^{n+1} )_{\bf F}$ has the following expansion:  
    \begin{equation} 
 ( \mu^{n+1} )_{\bf F} =    
   \varepsilon^{-1}  \Big(  {\cal I}_N ( \phi_{\bf F}^{n+1} )^3 
 - \hat{\phi}_{\bf F}^{n+1}  \Big) 
 - \varepsilon  ( \Delta_{h, (4)} \phi^{n+1} )_{\bf F} 
 - A \varepsilon^{-2} \dt ( \Delta_{h, (4)} ( \phi^{n+1} - \phi^n ) )_{\bf F}  , 
  \label{est-time-1-2} 
  \end{equation}
with $\hat{\phi}_{\bf F}^{n+1} =  2 \phi_{\bf F}^n - \phi_{\bf F}^{n-1}$. Moreover, with the help of the uniform in time estimates (\ref{LinftyHm-stability-bound}), combined with repeated applications of H\"older inequality and Sobolev embedding, we are able to derive the following estimates: 
\begin{align} 
  & 
  \nrm{ \Delta ( {\cal I}_N ( \phi_{\bf F}^{n+1} )^3 ) }_{H^k}  
  \le 
  C \nrm{ ( \phi_{\bf F}^{n+1} )^3 }_{H^{k+2} }  
  \le C  \nrm{ \phi_{\bf F}^{n+1} }_{H^{k+2} }^3 
   \le C \hat{C}_{k+2, \varepsilon}^3 , 
   \label{est-time-1-3} 
\\
  & 
   \| \Delta \hat{\phi}_{\bf F}^{n+1}  \|_{H^k} 
  \le C ( \nrm{ \phi_{\bf F}^n }_{H^{k+2} } 
   +  \nrm{ \phi_{\bf F}^{n-1} }_{H^{k+2} } ) 
   \le C \hat{C}_{k+2, \varepsilon} , 
   \label{est-time-1-4} 
\\
  & 
   \| \Delta ( \Delta_{h, (4)} \phi^\ell )_{\bf F}  \|_{H^k} 
   \le  \| ( \Delta_{h, (4)} \hat{\phi}^\ell )_{\bf F}  \|_{H^{k+2} }
   \le  \nrm{ \Delta \phi_{\bf F}^\ell  }_{H^{k+2} }  
     \nonumber 
\\
  &  \qquad \qquad  \quad 
  \le C \| \phi_{\bf F}^\ell \|_{H^{k+4} } 
   \le C \hat{C}_{k+4, \varepsilon} ,  \quad \ell = n, n+1 , 
   \label{est-time-1-5} 
\end{align}
In fact, inequality~\eqref{lem-1-4} has been applied again in the second step of (\ref{est-time-1-5}). In turn, a substitution of~\eqref{est-time-1-3}-\eqref{est-time-1-5} into \eqref{est-time-1-2} and \eqref{est-time-1-1} leads to  
\begin{equation} 
  \Big\| ( \frac32 \phi^{n+1}- 2 \phi^n + \frac12 \phi^{n-1} )_{\bf F} \Big\|_{H^k}  
   \le  D_{k, \varepsilon}^{(1)} \dt ,  \, \,   
   D_{k, \varepsilon}^{(1)} = C  \Big( 
   \varepsilon^{-1}  ( \hat{C}_{k+2, \varepsilon}^3 
    +  \hat{C}_{k+2, \varepsilon}  ) 
    + \varepsilon  \hat{C}_{k+4, \varepsilon}  \Big)  .  
   \label{est-time-1-6} 
    \end{equation} 
It is clear that $D_{k, \varepsilon}^{(1)}$ is a uniform-in-time constant, and it depends on $\varepsilon^{-1}$ in a polynomial form, since both $\hat{C}_{k+2, \varepsilon}$ and $\hat{C}_{k+4, \varepsilon}$ do. 

Meanwhile, a further estimate is needed to obtain a uniform estimate for $ \| ( \phi^{n+1}-  \phi^n )_{\bf F} \|_{H^k}$. Motivated by the triangular inequality 
\begin{equation} 
   \Big\| ( \frac32 \phi^{n+1}- 2 \phi^n + \frac12 \phi^{n-1} )_{\bf F} \Big\|_{H^k}  
   \ge \frac32 \| (  \phi^{n+1}-  \phi^n )_{\bf F} \|_{H^k}  
   - \frac12 \| ( \phi^n - \phi^{n-1} )_{\bf F} \|_{H^k} ,  \label{est-time-1-7-1} 
    \end{equation} 
we see that 
\begin{equation} 
   \| (  \phi^{n+1}-  \phi^n )_{\bf F} \|_{H^k}   
   \le \frac23 D_{k, \varepsilon}^{(1)} \dt + \frac13 \| ( \phi^n - \phi^{n-1} )_{\bf F} \|_{H^k} . 
   \label{est-time-1-7-2} 
    \end{equation} 
In turn, a recursive argument to the above inequality implies that 
\begin{equation} 
   \| (  \phi^{n+1}-  \phi^n )_{\bf F} \|_{H^k}   
   \le D_{k, \varepsilon}^{(1)} \dt + ( \frac13 )^n \| ( \phi^0 - \phi^{-1} )_{\bf F} \|_{H^k}  
   \le ( D_{k, \varepsilon}^{(1)} + C^* ) \dt ,  \quad \forall n \ge 0 , \label{est-time-1-7-3} 
\end{equation} 
in which $C^*$ is related to the initialization data~\eqref{scheme-BDF-CH-initial-1}. This finishes the proof of Theorem~\ref{stab-time-1}, by taking $\hat{D}_{k, \varepsilon} = D_{k, \varepsilon}^{(1)} + C^*$. 
\end{proof}

\subsection{A uniform estimate for $\nrm{ \phi^{n+1} - 2 \phi^n + \phi^{n-1} }_{H_h^k}$ } 

In addition to the first order temporal difference operator, the second order differential stencil term, namely $\phi^{n+1} - 2 \phi^n + \phi^{n-1}$, needs to be analyzed in the BDF2 scheme. A uniform-in-time estimate is presented in the following theorem. 

	\begin{thm} \label{stab-time-2}  
For the proposed numerical scheme~\eqref{scheme-BDF-CH-1}, if $\dt \le C \varepsilon^\frac12$,  we have 	
	\begin{equation}
\max_{1\le n \le M-1} \nrm{ \phi^{n+1}_{\bf F} 
- 2 \phi^n_{\bf F} + \phi^{n-1}_{\bf F} }_{H^k} 
 \le \hat{Q}_{k, \varepsilon} \dt^2 ,  \quad 
  \mbox{with} \, \,  \hat{Q}_{k, \varepsilon} 
  := C \varepsilon^{- m_k } ,
	\label{time-stability-2-bound}
	\end{equation}
where $m_k$ is a positive integer and $C>0$ is a constant independent of $h$, $\dt$, $T$, and $\varepsilon$. 
	\end{thm}
	
\begin{proof}
Taking a difference between $\frac32 \phi_{\bf F}^{n+1} - 2 \phi_{\bf F}^n + \frac12 \phi_{\bf F}^{n-1}$ and $\frac32 \phi_{\bf F}^n - 2 \phi_{\bf F}^{n-1} + \frac12 \phi_{\bf F}^{n-2}$, evaluated by the numerical algorithm~\eqref{scheme-BDF-CH-1} at time steps $t^{n+1}$, $t^n$, respectively, gives   
 	\begin{align} 
	 & 
 \Big\| ( \frac32 \phi^{n+1} - \frac72 \phi^n + \frac52 \phi^{n-1}
  - \frac12 \phi^{n-2} )_{\bf F} \Big\|_{H^k} 
  = 
  \dt \| \left( \Delta_{h, (4)} ( \mu^{n+1} 
   - \mu^n ) \right)_{\bf F} \|_{H^k}  \nonumber 
\\
  \le& 
    \dt \| \Delta ( \mu^{n+1} - \mu^n )_{\bf F} \|_{H^k} ,  
    \quad \mbox{with} \label{est-time-2-1}  
\\ 
  & 
   \mu^{n+1} - \mu^n = 
   \varepsilon^{-1}  \Big( 
   ( \phi^{n+1} )^3  - ( \phi^n )^3    
 -  2 ( \phi^n - \phi^{n-1} ) +  ( \phi^{n-1} - \phi^{n-2} )  \Big)  \nonumber 
\\
  &  \qquad \qquad \quad 
 - \varepsilon \Delta_{h, (4)} ( \phi^{n+1} - \phi^n )  
 - A \varepsilon^{-2} \dt ( \Delta_{h, (4)} ( \phi^{n+1} - \phi^n )  
 - \Delta_{h, (4)} ( \phi^n - \phi^{n-1} ) )  , 
  \label{est-time-2-2}  
\\
  & 
  ( \phi^{n+1} )^3  - ( \phi^n )^3 
  =  ( (\phi^{n+1})^2 + \phi^{n+1} \phi^n + (\phi^n)^2 ) 
   ( \phi^{n+1} - \phi^n )  .  \label{est-time-2-3}   
    \end{align} 
It is observed that all the right-hand-side terms in~\eqref{est-time-2-2} contain a factor of $\phi^\ell - \phi^{\ell-1}$, $n-1 \le \ell \le n+1$. In turn, with the help of \eqref{LinftyHm-stability-bound} and \eqref{time-stability-1-bound}, the following estimate could be carefully derived; the technical details are skipped for the sake of brevity: 
\begin{equation} 
   \Big\| ( \frac32 \phi^{n+1} - \frac72 \phi^n + \frac52 \phi^{n-1}
  - \frac12 \phi^{n-2} )_{\bf F} \Big\|_{H^k} 
    \le  \dt \| \Delta ( \mu^{n+1} - \mu^n )_{\bf F} \|_{H^k} 
     \le Q_{k, \varepsilon}^{(1)} \dt^2 , \label{est-time-2-4}   
\end{equation} 
with the uniform-in-time constant $Q_{k, \varepsilon}^{(1)}$ dependent on $\varepsilon^{-1}$ in a polynomial form. 

To obtain a useful bound for $\phi^{n+1} - 2 \phi^n + \phi^{n-1}$, we begin with the following observation: 
\begin{equation} 
\begin{aligned} 
  & 
  \Big\| ( \frac32 \phi^{n+1} - \frac72 \phi^n + \frac52 \phi^{n-1}
  - \frac12 \phi^{n-2} )_{\bf F} \Big\|_{H^k}  
\\
  = & 
  \Big\|  \frac32 ( \phi^{n+1} - 2 \phi^n + \phi^{n-1} )_{\bf F} 
  - \frac12 ( \phi^n - 2 \phi^{n-1} + \phi^{n-2} )_{\bf F} \Big\|_{H^k} 
\\
  \ge & 
  \frac32 \|  ( \phi^{n+1} - 2 \phi^n + \phi^{n-1} )_{\bf F}  \|_{H^k} 
  - \frac12 \| ( \phi^n - 2 \phi^{n-1} + \phi^{n-2} )_{\bf F} \|_{H^k}  . 
\end{aligned} 
  \label{est-time-2-5}   
\end{equation} 
Its combination with~\eqref{est-time-2-4} results in 
\begin{equation} 
   \| (  \phi^{n+1} -  2 \phi^n + \phi^{n-1} )_{\bf F} \|_{H^k}   
   \le \frac23 Q_{k, \varepsilon}^{(1)} \dt^2 
   + \frac13 \| ( \phi^n - 2 \phi^{n-1} + \phi^{n-2} )_{\bf F} \|_{H^k} . 
   \label{est-time-2-6} 
    \end{equation} 
Similarly, a recursive argument to this inequality leads to 
\begin{equation} 
   \| (  \phi^{n+1} -  2 \phi^n + \phi^{n-1} )_{\bf F} \|_{H^k}   
   \le Q_{k, \varepsilon}^{(1)} \dt^2 + ( \frac13 )^{n-1} \| ( \phi^1 - 2 \phi^0 + \phi^{-1} )_{\bf F} \|_{H^k}  
   \le ( Q_{k, \varepsilon}^{(1)} + C^{**} ) \dt^2 , \label{est-time-2-7} 
\end{equation} 
for any $n \ge 1$, in which $C^{**}$ is related to the initialization data of $\phi^1$, $\phi^0$ and $\phi^{-1}$, with $\phi^{-1}$ gives by~\eqref{scheme-BDF-CH-initial-1}. This finishes the proof of Theorem~\ref{stab-time-2}, by taking $\hat{Q}_{k, \varepsilon} = Q_{k, \varepsilon}^{(1)} + C^{**}$. 
\end{proof}

\begin{rem} 
If the Crank-Nicolson-style temporal discretization is applied, the associated uniform-in-time estimates for $\| \phi^{n+1}_{\bf F} -  \phi^n_{\bf F} \|_{H^k}$ and $\| \phi^{n+1}_{\bf F} - 2 \phi^n_{\bf F} + \phi^{n-1}_{\bf F} \|_{H^k}$ could be carried out in a more straightforward way, because of the simpler temporal stencil of $\phi^{n+1} - \phi^n$; see the detailed analysis in~\cite{guo21}. In comparison, if a BDF temporal stencil is applied, a direct control of $\| \phi^{n+1}_{\bf F} -  \phi^n_{\bf F} \|_{H^k}$ is not available any more. Instead, we have to make use of the triangular inequalities~\eqref{est-time-1-7-1}, \eqref{est-time-2-5}, which come from the original numerical algorithm. These inequalities in turn enable us to derive the desired uniform-in-time bounds~\eqref{est-time-1-7-3}, \eqref{est-time-2-7}, using a recursive argument. This subtle technique could be applied to many other BDF-style, energy stable numerical schemes for various gradient flow models~\cite{cheng2019b, fengW18b, Hao2021, Hou2023, LiW18, Yuan2022a}, etc. 
\end{rem}

\subsection{Some established estimates for the exact solution and projection solution} 
	
For the exact solution $\Phi$ and the projection solution $\Phi_N$, the following estimates could be derived by the energy estimate and Sobolev analysis; also see the relevant reference work~\cite{guo21}. 

	\begin{thm} \cite{guo21} \label{stab-exact-1}  
The following estimates are valid for the exact solution $\Phi$: 
	\begin{align}
	  & 
\nrm{\Phi }_{L^\infty(0,T;H^{m_0})} 
 \le \hat{C}_{m_0, \varepsilon}^* 
  := C \varepsilon^{- k_{m0} } ,  
   \quad \forall m_0 \ge 1 , 
	\label{LinftyHm-stability-exact-bound} 
\\
  & 
  \nrm{\partial_t \Phi }_{L^\infty(0,T;H^k)} 
 \le \hat{D}_{k, \varepsilon}^* ,  \quad 
  \mbox{with} \, \,  \hat{D}_{k, \varepsilon}^* 
  := C \varepsilon^{- n_k } , 
   \quad \forall k \ge 0 , 
	\label{time-stability-1-exact-bound} 
\\
  & 
  \nrm{\partial_t^2 \Phi }_{L^\infty(0,T;H^k)} 
 \le \hat{Q}_{k, \varepsilon}^* ,  \quad 
  \mbox{with} \, \,  \hat{Q}_{k, \varepsilon}^* 
  := C \varepsilon^{- m_k } , 
   \quad \forall k \ge 0 , 
	\label{time-stability-2-exact-bound} 
\\
  	  & 
  \max_{0\le n \le M} \nrm{\Phi_N^n }_{H^{m_0}} 
 \le \hat{C}_{m_0, \varepsilon}^* 
  := C \varepsilon^{- k_{m0} } , 
	\label{LinftyHm-stability-projection-bound} 
\\
  & 
  \max_{0\le n \le M-1} 
   \nrm{\Phi_N^{n+1} - \Phi_N^n }_{H^k} 
 \le \hat{D}_{k, \varepsilon}^* \dt ,  \quad 
  \hat{D}_{k, \varepsilon}^* 
  := C \varepsilon^{- n_k } , 
	\label{time-stability-1-projection-bound} 
\\
  & 
  \max_{1\le n \le M-1} 
   \nrm{\Phi_N^{n+1} - 2 \Phi_N^n + \Phi_N^{n-1} }_{H^k} 
 \le \hat{Q}_{k, \varepsilon}^* \dt^2 ,  \quad 
   \hat{Q}_{k, \varepsilon}^* 
  := C \varepsilon^{- m_k } , 
	\label{time-stability-2-projection-bound} 
	\end{align}
where $k_{m0}$, $n_k$ and $m_k$ are given integers and $C$ is a constant independent of $h$, $\dt$, $T$, and $\varepsilon$. 
	\end{thm}


	\section{Error analysis with a refined convergence constant}  
	\label{sec:convergence}

With the help of the preliminary estimates for the numerical solution, established in the last section, we are able to establish an improved convergence analysis.

\subsection{Consistency analysis and the numerical error evolutionary equation} 
\label{subsec:consistency}

For the projection solution $\Phi_N$, the following estimate is valid: 
\begin{equation} 
  \partial_t \Phi_N = \Delta \left( 
  \varepsilon^{-1} ( \Phi_N^3 - \Phi_N ) 
  - \varepsilon \Delta \Phi_N  \right) 
  + \tau_0 ,  \quad \mbox{with} \, \, 
  \nrm{\tau_0 (t)}_{H^{-1}_{\rm per}} \le C h^m \varepsilon^{- j_1} , 
  \label{consistency-1} 
\end{equation} 
in which the projection approximation estimates~\eqref{projection-est-1}, \eqref{projection-est-4}, combined with the exact CH equation~\eqref{CH equation}, have to be applied in the derivation. 

If the fourth order long stencil difference approximation is taken in space, the following estimate is available: 
\begin{align} 
  & 
  \partial_t \Phi_N = \Delta_{h, (4)} \left( 
  \varepsilon^{-1} ( \Phi_N^3 - \Phi_N ) 
  - \varepsilon \Delta_{h, (4)} \Phi_N  \right) 
  + \tau_1 ,   \mbox{with}  \quad  
  \| \tau_1 (t) \|_{-1, h} \le C h^4 \varepsilon^{- j_2}  . 
  \nonumber  
\end{align} 
Moreover, by taking a modified BDF2 temporal discretization, we have 
\begin{align} 
  \frac{\frac32 \Phi_N^{n+1} - 2 \Phi_N^n + \frac12 \Phi_N^{n-1}}{\dt}  
  = & \Delta_{h, (4)} \Big( 
  \varepsilon^{-1} (  ( \Phi_N^{n+1} )^3 - \hat{\Phi}_N^{n+1}  )  
  - \varepsilon \Delta_{h, (4)}  \Phi_N^{n+1}    \nonumber 
\\
  &  \qquad \quad 
  + A \varepsilon^{-2} \dt \Delta_{h, (4)} ( \Phi_N^{n+1} - \Phi_N^n ) \Big) 
  + \tau^{n+1} ,    
   \label{consistency-3}
\\
  \mbox{with}  \quad  &
  \hat{\Phi}_N^{n+1} = 2 \Phi_N^n - \Phi_N^{n-1} , \,  \, \,  
  \quad \| \tau^{n+1} \|_{-1, h} 
  \le C (\dt^2 + h^4 )  \varepsilon^{- j_3} . 
  \nonumber  
\end{align}

In turn, subtracting the consistency estimate~\eqref{consistency-3} from the numerical scheme~\eqref{scheme-BDF-CH-1} gives
\begin{align}  
  \frac{\frac32 e^{n+1} - 2 e^n + \frac12 e^{n-1}}{\dt}  
   =& \Delta_{h, (4)} \Big( 
  \varepsilon^{-1} (  ( \Phi_N^{n+1} )^3 - ( \phi^{n+1} )^3 
   - \hat{e}^{n+1}  ) 
  - \varepsilon \Delta_{h, (4)}  e^{n+1}      
  \nonumber 
\\
  & 
  + A \varepsilon^{-2} \dt \Delta_{h, (4)} ( e^{n+1} - e^n ) \Big) 
  + \tau^{n+1} ,  \label{eq:error.1} 
\\ 
  \mbox{with}  \quad &  
  \hat{e}^{n+1} = 2 e^n - e^{n-1} ,  \quad 
  \| \tau^{n+1} \|_{-1, h} 
  \le C ( \dt^2 + h^4 )  \varepsilon^{- j_4} . 
   \nonumber 
\end{align} 
	
Meanwhile, it is noticed that the numerical error functions are discrete, and they are only evaluated at the numerical grid points.  Using the tool of~\eqref{def:Fourier} and \eqref{def:extension}, we make a continuous extension of these discrete error functions: 
\begin{eqnarray} 
  e_{\bf F}^k = \Phi_N^k - \phi_{\bf F}^k ,  \quad  
    \hat{e}_{\bf F}^{k+1} 
  = 2 e_{\bf F}^k -  e_{\bf F}^{k-1} . 
   \label{error function-2} 
\end{eqnarray}

Meanwhile, a few preliminary estimates for the numerical error term will be needed in the later analysis. By making a comparison between~\eqref{LinftyHm-stability-bound}, \eqref{time-stability-1-bound}, \eqref{time-stability-2-bound} and \eqref{LinftyHm-stability-projection-bound}-\eqref{time-stability-2-projection-bound}, the following results become straightforward.

	\begin{lem}
	\label{lemma:prelim est}
For the numerical error function, we have 
	\begin{align}
		  & 
  \max_{0\le n \le M} \| e_{\bf F}^n \|_{H^{m_0}} 
 \le \hat{C}_{m_0, \varepsilon}^{**} 
  := C \varepsilon^{- k_{m0} } , 
	\label{convergence-prelim est-1} 
\\
  & 
  \max_{0\le n \le M-1} \| e_{\bf F}^{n+1} - e_{\bf F}^n \|_{H^k} 
 \le \hat{D}_{k, \varepsilon}^{**} \dt ,  \quad 
  \hat{D}_{k, \varepsilon}^{**} 
  := C \varepsilon^{- n_k } , 
	\label{convergence-prelim est-2} 
\\
  & 
  \max_{1\le n \le M-1} 
   \| e_{\bf F}^{n+1} - 2 e_{\bf F}^n + e_{\bf F}^{n-1} \|_{H^k} 
 \le \hat{Q}_{k, \varepsilon}^{**} \dt^2 ,  \quad 
   \hat{Q}_{k, \varepsilon}^{**} 
  := C \varepsilon^{- m_k } , 
	\label{convergence-prelim est-3}  
	\end{align}
for any $m_0 \ge 1$ and $k \ge 0$, where $k_{m0}$, $n_k$ and $m_k$ are given integers and $C$ is a constant independent of $h$, $\dt$, $T$, and $\varepsilon$. 
	\end{lem}
	
In fact, these preliminary bounds for the numerical error function do not rely on the error and convergence analysis, and all the constants are final time independent.

\subsection{Review of the spectrum estimate for the linearized Cahn-Hilliard operator} 

The following linearized spectrum estimate has been established in~\cite{alikakos94, alikakos93, chenx94, feng04}. We just recall this result. 

\begin{prop}  \label{Feng04NM}  (\cite{feng04}) 
  There exists $0 < \varepsilon_0 << 1$ and another positive constant $C_0$ such that the principle eigenvalue of the linearized Cahn-Hilliard operator satisfies 
\begin{eqnarray} 
  \lambda_{CH} := \inf_{\psi \in H^1, \psi \ne 0} 
  \frac{ \varepsilon^{-1}  \left( 
  \left( 3  \Phi^2 (t) - 1 \right) \psi , \psi \right) 
  + \varepsilon \nrm{ \nabla \psi }^2 } 
   { \nrm{ \psi }_{H^{-1} }^2 }  \ge - C_0 , 
   \label{spectrum-est-1} 
\end{eqnarray} 
for any $t \ge 0$, $0 < \varepsilon < \varepsilon_0$, where $\Phi$ is the exact solution to the Cahn-Hilliard problem. 
\end{prop} 


\subsection{Proof of Theorem~\ref{thm:convergence} } 
\label{subsec:error analysis} 

Since the numerical error function is mean free, i.e., $\overline{e^k} = 0$, its $\| \cdot \|_{-1, h}$ norm is well defined. Taking a discrete inner product with~\eqref{eq:error.1} by $2 (-\Delta_{h, (4)})^{-1} e^{n+1}$ yields   
\begin{align}
  &   
    \frac{1}{\dt} \Big\langle 3 e^{n+1} - 4 e^n + e^{n-1} ,  (-\Delta_{h, (4)})^{-1} e^{n+1}  \Big\rangle 
   +  2 \varepsilon^{-1} \langle  
    ( ( \Phi_N^{n+1} )^3 - ( \phi^{n+1} )^3 
   - \hat{e}^{n+1}  ) , e^{n+1}  \rangle    \nonumber 
\\
  & 
   - 2 \varepsilon \langle \Delta_{h, (4)} e^{n+1} , e^{n+1} \rangle 
   - 2 A \varepsilon^{-2} \langle \Delta_{h, (4)} ( e^{n+1} - e^n ) , e^{n+1} \rangle     
  = 2 \langle \tau^{n+1} ,  (-\Delta_{h, (4)})^{-1} e^{n+1} \rangle , 
   \label{convergence-est-1}
\end{align} 
in which the summation by parts formula has been repeatedly applied. 

The analysis for the temporal differentiation term is similar to~\eqref{H2-est-2}:   
\begin{align} 
  &  
   \langle 3 e^{n+1} - 4 e^n + e^{n-1} , (-\Delta_{h, (4)} )^{-1} e^{n+1} \rangle  
   = \langle   3 e^{n+1} - 4 e^n + e^{n-1}  ,  e^{n+1} \rangle_{-1, h}    
   \nonumber 
\\
  =&
     \frac12 ( \| e^{n+1} \|_{-1, h}^2 - \| e^n \|_{-1, h}^2 
   + \| 2 e^{n+1} - e^n \|_{-1, h}^2 
   - \| 2 e^n - e^{n-1}  \|_{-1, h}^2  \nonumber 
 \\
   & \quad 
   +  \| e^{n+1} - 2 e^n + e^{n-1}  \|_{-1, h}^2 ) . \label{convergence-est-2} 
\end{align} 

The local truncation error term could be bounded in a straightforward way: 
\begin{equation} 
\begin{aligned} 
  2 \langle \tau^{n+1} ,  (-\Delta_{h, (4)})^{-1} e^{n+1} \rangle 
  = & 2 \langle \tau^{n+1} ,  e^{n+1} \rangle_{-1, h}  
  \le 2 \| \tau^{n+1} \|_{-1, h} \cdot \|  e^{n+1} \|_{-1, h}  
\\
  \le & \| \tau^{n+1} \|_{-1, h}^2 + \|  e^{n+1} \|_{-1, h}^2 .  
\end{aligned} 
    \label{convergence-est-3}
\end{equation} 

For the surface diffusion term, we apply the summation-by-parts formula~\eqref{summation-1}, as well as inequality~\eqref{lem-2-1} (in Lemma~\ref{lemma:2}), and get 
\begin{align} 
   & 
   - \langle \Delta_{h, (4)} e^{n+1} , e^{n+1} \rangle  
   = \| \nabla_{h, (4)} e^{n+1} \|_2^2 ,  \label{convergence-est-4-1} 
\\
  & 
   \| \nabla e_{\bf F}^{n+1} \|^2 
   - \| \nabla_{h, (4)} e^{n+1} \|_2^2 \le C h^4  \| e_{\bf F}^{n+1} \|_{H^3}^2 . 
  \label{convergence-est-4-2}
\end{align} 
On the other hand, the following Sobolev interpolation inequality could be used to obtain a sharper bound on the right hand side: 
\begin{equation} 
  \| e_{\bf F}^{n+1} \|_{H^3}^2 
  \le  C  \| \nabla \Delta e_{\bf F}^{n+1} \|^2 
  \le C  \| e_{\bf F}^{n+\hf} \|_{H^7} 
   \cdot \| e_{\bf F}^{n+\hf} \|_{H^{-1} }  
   \le  
   C  \hat{C}_{7, \varepsilon}^{**}  \| e_{\bf F}^{n+\hf} \|_{H^{-1} } , 
   \label{convergence-est-4-3}
\end{equation}
with the preliminary estimate~\eqref{convergence-prelim est-1} (in Lemma~\ref{lemma:prelim est}) applied in the last step. Then we obtain 
\begin{align} 
   & 
     \| \nabla e_{\bf F}^{n+1} \|^2 
   - \| \nabla_{h, (4)} e^{n+1} \|_2^2  
  \le C \hat{C}_{7, \varepsilon}^{**} h^4 
   \nrm{ e_{\bf F}^{n+1} }_{H^{-1} }   
    \le  \| e_{\bf F}^{n+1} \|_{H^{-1} }^2 
    +  C h^8 ( \hat{C}_{7, \varepsilon}^{**} )^2 , \label{convergence-est-4-4} 
\\
  &  \mbox{so that}  \quad 
   - \langle \Delta_{h, (4)} e^{n+1} , e^{n+1} \rangle   
   \ge \| \nabla e_{\bf F}^{n+1} \|^2  -  C h^8 ( \hat{C}_{7, \varepsilon}^{**} )^2 . 
   \label{convergence-est-4-5} 
\end{align} 

Similar to~\eqref{H2-est-3}, the artificial regularization inner product term satisfies the following triangular equality:  
	\begin{equation}
- 2 \langle \Delta_{h, (4)} ( e^{n+1} - e^n ) , e^{n+1}  \rangle   
=  \| \nabla_{h, (4)} e^{n+1} \|_2^2 - \| \nabla_{h, (4)} e^n \|_2^2  
 +    \| \nabla_{h, (4)} ( e^{n+1} - e^n)  \|_2^2   .    
	\label{convergence-est-5}
	\end{equation}
	
For the concave term, we begin with the following observation: 
\begin{equation} 
\begin{aligned} 
  \langle  - \hat{e}^{n+1} , e^{n+1}  \rangle  
  = & - \| e^{n+1} \|_2^2 + \langle  e^{n+1} - 2 e^n + e^{n-1} , e^{n+1}  \rangle  
\\
  = & 
    - \| e_{\bf F}^{n+1} \|^2 
  + ( e_{\bf F}^{n+1} - 2 e_{\bf F}^n + e_{\bf F}^{n-1} , e_{\bf F}^{n+1}  ) , 
\end{aligned} 
  \label{convergence-est-6-1}
\end{equation} 
in which identity~\eqref{lem-3-1} (in Lemma~\ref{lemma:3}) has been applied in the second step. Meanwhile, the second term on the right hand side could be analyzed as follows: 
\begin{equation} 
\begin{aligned} 
  | ( e_{\bf F}^{n+1} - 2 e_{\bf F}^n + e_{\bf F}^{n-1} , e_{\bf F}^{n+1}  ) | 
  \le & \| \nabla ( e_{\bf F}^{n+1} - 2 e_{\bf F}^n + e_{\bf F}^{n-1} ) \| 
  \cdot \| e_{\bf F}^{n+1} \|_{H^{-1}}  
\\
  \le & 
  \hat{Q}_{1, \varepsilon}^{**} \dt^2 \cdot \| e_{\bf F}^{n+1} \|_{H^{-1}}   
  \le \frac{\varepsilon}{2} \| e_{\bf F}^{n+1} \|_{H^{-1}}^2 
  + \frac12 \varepsilon^{-1} ( \hat{Q}_{1, \varepsilon}^{**} )^2 \dt^4 , 
\end{aligned} 
  \label{convergence-est-6-2}
\end{equation} 
where the second step comes from the preliminary estimate~\eqref{convergence-prelim est-3} (in Lemma~\ref{lemma:prelim est}). Then we arrive at a bound for the concave term: 
\begin{equation} 
   2 \varepsilon^{-1} \langle  - \hat{e}^{n+1} , e^{n+1}  \rangle  
   \ge - 2 \varepsilon^{-1} \| e_{\bf F}^{n+1} \|^2 
   -  \| e_{\bf F}^{n+1} \|_{H^{-1}}^2 
  - \varepsilon^{-2} ( \hat{Q}_{1, \varepsilon}^{**} )^2 \dt^4 . 
   \label{convergence-est-6-3}
\end{equation}

The rest work is focused on the nonlinear estimate. The following a-priori assumption is made. 

\noindent
{\bf An a-priori assumption up to time step $t^n$.}  \, \, We assume a-priori that the numerical error function has the following convergence order, 
at time steps up to $t^n$:     
\begin{equation} 
  \| e^\ell \|_{-1, h}  \le \dt^{\frac{15}{8}} + h^{\frac{15}{4}} , \quad 
     \ell \le n .    
  \label{convergence-a priori-1}
\end{equation} 

First, the nonlinear error inner product could be expanded as 
\begin{equation} 
  I_1^{(d)} := \langle  ( \Phi_N^{n+1} )^3 - ( \phi^{n+1} )^3 , e^{n+1}  \rangle 
    =  \Big\langle  
   \Big( ( \Phi_N^{n+1} )^2 + \Phi_N^{n+1} \phi^{n+1}  
   + ( \phi^{n+1} )^2 \Big) ,  ( e^{n+1} )^2 \Big\rangle . 
     \label{convergence-est-7}
\end{equation} 
To bound this nonlinear inner product, we make use of~\eqref{lem-3-2} (in Lemma~\ref{lemma:3}) to control its difference with its continuous version: 
\begin{align} 
  & 
  I_1 := \Big(   
   \Big( ( \Phi_N^{n+1} )^2 + \Phi_N^{n+1} \phi_{\bf F}^{n+1}  
   + ( \phi_{\bf F}^{n+1} )^2 \Big) , ( e_{\bf F}^{n+1} )^2 \Big) ,  
       \label{convergence-est-8-1}
\\
  & 
  | I_1^{(d)} - I_1 |  
  \le C h^8  \Big(   
   \| ( \Phi_N^{n+1} )^2 + \Phi_N^{n+1} \phi_{\bf F}^{n+1}  
   + ( \phi_{\bf F}^{n+1} )^2  \|_{H^8} 
   \cdot \|  ( e_{\bf F}^{n+1} )^2 \|_{H^8}  
   \Big)   \nonumber 
\\
  &  \quad 
   \le C h^8  \Big( 
    \Big( \nrm{ \Phi_N^{n+1} }_{H^8}^2 
   +  \nrm{ \phi_{\bf F}^{n+1} }_{H^8}^2 \Big) 
   \cdot  \nrm{ e_{\bf F}^{n+1} }_{H^8}^2  \Big) 
   \nonumber 
\\
  & \quad 
  \le C h^8  \Big( 
      \hat{C}_{8, \varepsilon}^4
      + ( \hat{C}_{8, \varepsilon}^* )^4
      + ( \hat{C}_{8, \varepsilon}^{**} )^4  \Big)  
    \le  C h^8  \varepsilon^{-4 k_8} ,  
    \label{convergence-est-8-2}
\end{align} 
in which the established estimates (\ref{LinftyHm-stability-bound}),  (\ref{LinftyHm-stability-projection-bound})  and (\ref{convergence-prelim est-1}) have been repeatedly applied. 

To estimate the continuous inner product $I_1$ in~\eqref{convergence-est-8-1}, the following identity is observed:  
\begin{equation} 
   ( \Phi_N^{n+1} )^2 
   + \Phi_N^{n+1} \phi_{\bf F}^{n+1}  
   + ( \phi_{\bf F}^{n+1} )^2  
     = 3 ( \Phi_N^{n+1} )^2 - 3 \Phi_N^{n+1} e_{\bf F}^{n+1}  
   + ( e_{\bf F}^{n+1} )^2 .   
   \label{convergence-est-8-3}
\end{equation} 
This in turn gives 
\begin{equation} 
  I_1 \ge  
   3 \Big( ( \Phi_N^{n+1} )^2  ,  ( e_{\bf F}^{n+1} )^2 \Big)  
   + {\cal IE} ,  \quad {\cal IE} = 
   - 3 \Big( \Phi_N^{n+1} e_{\bf F}^{n+1} , 
      ( e_{\bf F}^{n+1} )^2 \Big) . 
      \label{convergence-est-8-4} 
\end{equation}
In terms of the additional nonlinear error inner product term, an application of H\"older inequality, combined with Sobolev inequality, indicates that 
\begin{equation} 
  \left| {\cal IE} \right| 
  \le 3  \| \Phi_N^{n+1} \|_{L^\infty}  
      \cdot \| e_{\bf F}^{n+1} \|_{L^3}^3 
  \le C  \| \Phi_N^{n+1} \|_{H^2}  
      \cdot \| e_{\bf F}^{n+1} \|_{L^3}^3  
      \le  C \hat{C}_{2, \varepsilon}^*   
       \| e_{\bf F}^{n+1} \|_{L^3}^3 , 
      \label{convergence-est-8-5} 
\end{equation}
with estimate~\eqref{LinftyHm-stability-projection-bound} applied in the last step. To control $\| e_{\bf F}^{n+1} \|_{L^3}$, we see that 
\begin{align}  
  e_{\bf F}^{n+1} 
   =& \hat{e}_{\bf F}^{n+1}  
    + (  e_{\bf F}^{n+1} - 2 e_{\bf F}^n + e_{\bf F}^{n-1} )  ,  \quad 
    \mbox{so that}  \nonumber  
\\ 
  \| e_{\bf F}^{n+1}  \|_{L^3}^3 
   \le& C \Big( 
    \| \hat{e}_{\bf F}^{n+1} \|_{L^3}^3 
    + \|   e_{\bf F}^{n+1} - 2 e_{\bf F}^n + e_{\bf F}^{n-1}  \|_{L^3}^3 \Big)  \nonumber  
\\
   \le& C \Big( 
    \| \hat{e}_{\bf F}^{n+1} \|_{L^3}^3 
    + \|   e_{\bf F}^{n+1} - 2 e_{\bf F}^n + e_{\bf F}^{n-1}  \|_{H^1}^3 \Big)     
     \le C \Big( 
     \| \hat{e}_{\bf F}^{n+1} \|_{L^3}^3  
    + \dt^6 ( \hat{Q}_{1,\varepsilon}^{**} )^3  \Big) , 
    \label{convergence-est-8-6} 
\end{align} 
with the preliminary estimate \eqref{convergence-prelim est-3} applied in the last step. Meanwhile, the term $\| \hat{e}_{\bf F}^{n+1} \|_{L^3}$ could be analyzed with the help of the Sobolev inequalities:  
\begin{align} 
  \| \hat{e}_{\bf F}^{n+1} \|_{L^3} 
   \le&  C \| \hat{e}_{\bf F}^{n+1} \|_{H^\frac12} 
   \le  C  \| \hat{e}_{\bf F}^{n+1} \|_{H^{-1}}^\frac34 
   \cdot  \| \hat{e}_{\bf F}^{n+1} \|_{H^5}^\frac14   \nonumber 
\\
  \le& 
  C ( \hat{C}_{5,\varepsilon}^{**} )^{\frac14}  
   \cdot \| \hat{e}_{\bf F}^{n+1} \|_{H^{-1}}^{\frac34}  
  \le C ( \hat{C}_{5,\varepsilon}^{**} )^{\frac14}  
   \left( \dt^{\frac{15}{8}} + h^{\frac{15}{4}} \right)^{\frac34} .   
   \label{convergence-est-8-7} 
\end{align} 
Similarly, the preliminary estimate~\eqref{convergence-prelim est-1} has been applied in the third step and the a-priori assumption~\eqref{convergence-a priori-1} has been recalled in the last step, since $\hat{e}^{n+1}$ is only involved with the numerical errors at $t^n$ and $t^{n-1}$. In turn, a combination of~\eqref{convergence-est-8-6}-\eqref{convergence-est-8-7} and \eqref{convergence-est-8-5} leads to 
\begin{eqnarray} 
  \left| {\cal IE} \right|  \le 
    \hat{R}_{2, \varepsilon}^*  
     \Big( \dt^{\frac{135}{32}} + h^{\frac{135}{16}} \Big)   
    + \hat{R}_{3, \varepsilon}^* \dt^6  , 
      \label{convergence-est-8-8} 
\end{eqnarray}
with $\hat{R}_{2, \varepsilon}^* = C \hat{C}_{2, \varepsilon}^* ( \hat{C}_{5,\varepsilon}^{**} )^{\frac34}$, $\hat{R}_{3, \varepsilon}^* = C \hat{C}_{2, \varepsilon}^* ( \hat{Q}_{1,\varepsilon}^{**} )^3$. Going back~\eqref{convergence-est-8-4}, we get  
\begin{equation} 
  I_1 \ge  
   3 \Big( ( \Phi_N^{n+1} )^2  , ( e_{\bf F}^{n+1} )^2 \Big) 
     - \hat{R}_{2, \varepsilon}^*  
     \Big( \dt^{\frac{135}{32}} + h^{\frac{135}{16}} \Big)   
    - \hat{R}_{3, \varepsilon}^* \dt^6 . 
    \label{convergence-est-8-9}
\end{equation}

Therefore, a substitution of~\eqref{convergence-est-2}, \eqref{convergence-est-3}, \eqref{convergence-est-4-5}, \eqref{convergence-est-5}, \eqref{convergence-est-6-3}, \eqref{convergence-est-7}, \eqref{convergence-est-8-2} and \eqref{convergence-est-8-9} into \eqref{convergence-est-1} yields  
\begin{align} 
  & 
     \frac{1}{2 \dt} \Big( \| e^{n+1} \|_{-1, h}^2 - \| e^n \|_{-1, h}^2 
   + \| 2 e^{n+1} - e^n \|_{-1, h}^2 
   - \| 2 e^n - e^{n-1}  \|_{-1, h}^2 \Big)  \nonumber 
\\
  & 
    + A \varepsilon^{-1} \dt ( \| \nabla_{h, (4)} e^{n+1} \|_2^2 
    - \| \nabla_{h, (4)} e^n \|_2^2 ) 
     + 2 \Big( \varepsilon^{-1} \Big( 
     3 ( \Phi_N^{n+1} )^2  - 1  , 
      ( e_{\bf F}^{n+1} )^2 \Big)  
      + \varepsilon \| \nabla e_{\bf F}^{n+1} \|^2  \Big) \nonumber 
\\
  \le& 
    \| \tau^{n+1} \|_{-1, h}^2 + \|  e^{n+1} \|_{-1, h}^2 
    + \|  e_{\bf F}^{n+1} \|_{H^{-1}}^2   
    +  C ( \varepsilon ( \hat{C}_{7, \varepsilon}^{**} )^2 + \varepsilon^{-4 k_8} ) h^8  \nonumber 
\\
  & 
    + \varepsilon^{-2} ( \hat{Q}_{1, \varepsilon}^{**} )^2 \dt^4   
   + 2 \varepsilon^{-1} \Big( \hat{R}_{2, \varepsilon}^*  
     \Big( \dt^{\frac{135}{32}} + h^{\frac{135}{16}} \Big)   
    + \hat{R}_{3, \varepsilon}^* \dt^6 \Big) . 
      \label{convergence-est-9-1}
\end{align} 
On the other hand, a careful application of the linearized spectrum estimate~\eqref{spectrum-est-1} (reviewed in Proposition~\ref{Feng04NM}, by~\cite{feng04}) reveals that 
\begin{equation} 
    \varepsilon^{-1} \Big( 
     3 ( \Phi_N^{n+1} )^2  - 1  , 
      ( e_{\bf F}^{n+1} )^2 \Big)  
      + \varepsilon \| \nabla e_{\bf F}^{n+1} \|^2  
      \ge 
    - C_0 \| e_{\bf F}^{n+1} \|_{H^{-1} }^2 
    \ge  - C_0 \| e^{n+1} \|_{-1, h}^2 , 
    \label{convergence-est-9-2}
\end{equation} 
in which inequality~\eqref{lem-1-3} (in Lemma~\ref{lemma:1}) has been applied at the last step. Subsequently, with an introduction of error norm quantity 
\begin{equation} 
  E^n := \frac12 ( \| e^n \|_{-1, h}^2 + \| 2 e^n - e^{n-1}  \|_{-1, h}^2 )  
    + A \varepsilon^{-1} \dt^2 \| \nabla_{h, (4)} e^n \|_2^2 , 
    \label{defi-E-1} 
\end{equation} 
we obtain  
\begin{align} 
     \frac{1}{\dt} ( E^{n+1} - E^n ) 
  \le& 
    \| \tau^{n+1} \|_{-1, h}^2 + ( 2 C_0 + 1) \|  e^{n+1} \|_{-1, h}^2 
    + \|  e_{\bf F}^{n+1} \|_{H^{-1}}^2 
    +  C ( \varepsilon ( \hat{C}_{7, \varepsilon}^{**} )^2 + \varepsilon^{-4 k_8} ) h^8  \nonumber 
\\
  & 
    + \varepsilon^{-2} ( \hat{Q}_{1, \varepsilon}^{**} )^2 \dt^4  
   + 2 \varepsilon^{-1} \Big( \hat{R}_{2, \varepsilon}^*  
     \Big( \dt^{\frac{135}{32}} + h^{\frac{135}{16}} \Big)   
    + \hat{R}_{3, \varepsilon}^* \dt^6  \Big) \nonumber 
\\
   \le& 
     ( 2 C_0 + 2) \|  e^{n+1} \|_{-1, h}^2 + C ( \dt^4 + h^8 ) \varepsilon^{-2 j_4} 
    +  C ( \varepsilon ( \hat{C}_{7, \varepsilon}^{**} )^2 + \varepsilon^{-4 k_8} ) h^8 
     \nonumber 
\\
  & 
    + \varepsilon^{-2} ( \hat{Q}_{1, \varepsilon}^{**} )^2 \dt^4  
   + 2 \varepsilon^{-1} \Big( \hat{R}_{2, \varepsilon}^*  
     \Big( \dt^{\frac{135}{32}} + h^{\frac{135}{16}} \Big)   
    + \hat{R}_{3, \varepsilon}^* \dt^6 \Big)   \nonumber 
\\
   \le& 
     ( 4 C_0 + 4) E^{n+1} + C ( \dt^4 + h^8 ) 
     ( \varepsilon^{-2 j_4} 
    +  \varepsilon ( \hat{C}_{7, \varepsilon}^{**} )^2 + \varepsilon^{-4 k_8}   
    + \varepsilon^{-2} ( \hat{Q}_{1, \varepsilon}^{**} )^2 ) 
     \nonumber 
\\
  & 
   + 2 \varepsilon^{-1} \Big( \hat{R}_{2, \varepsilon}^*  
     \Big( \dt^{\frac{135}{32}} + h^{\frac{135}{16}} \Big)   
    + \hat{R}_{3, \varepsilon}^* \dt^6 \Big)  ,   
      \label{convergence-est-9-3}
\end{align} 
in which inequality~\eqref{lem-1-3} was applied in the second step, and the third step comes from the fact that $\| e^{n+1} \|_{-1, h}^2 \le 2 E^{n+1}$. In addition, under the following constraint for the time step and space mesh size:  
\begin{align} 
  &
   \hat{R}_{2, \varepsilon}^* \varepsilon^{-1}   
     \dt^{\frac{7}{32}} \le \frac14 ,  \quad 
     \hat{R}_{2, \varepsilon}^* \varepsilon^{-1}   
     h^{\frac{7}{16}}  \le \frac12 ,  \quad 
      \hat{R}_{3, \varepsilon}^* \varepsilon^{-1}   
     \dt^2 \le \frac14,  \nonumber  
\\
  &
  \mbox{i.e.} \quad 
  \dt , h \le  \min \Big( \Big( \frac{\varepsilon}{4} 
  ( \hat{R}_{2, \varepsilon}^* )^{-1} \Big)^{\frac{32}{7}} , 
   \Big( \frac{\varepsilon}{4} 
  ( \hat{R}_{3, \varepsilon}^* )^{-1} \Big)^{\frac12}  \Big) , 
   \label{convergence-condition-2} 
\end{align} 
the following inequality is valid: 
\begin{equation} 
\begin{aligned} 
  & 
     \frac{1}{\dt} ( E^{n+1} - E^n ) 
    \le ( 4 C_0 + 4) E^{n+1} + \hat{R}_{4, \varepsilon}^* ( \dt^4 + h^8 ) ,  
\\
  & 
     \hat{R}_{4, \varepsilon}^* = C ( \varepsilon^{-2 j_4} 
    +  \varepsilon ( \hat{C}_{7, \varepsilon}^{**} )^2 + \varepsilon^{-4 k_8}   
    + \varepsilon^{-2} ( \hat{Q}_{1, \varepsilon}^{**} )^2 ) + 1 . 
\end{aligned} 
      \label{convergence-est-9-4}
\end{equation} 
Of course,  $4 C_0 +4$ is a constant independent on $\varepsilon$, and $\hat{R}_{4, \varepsilon}^*$ depends on $\varepsilon^{-1}$ in a polynomial form. In turn, an application of discrete Gronwall inequality to~\eqref{convergence-est-9-4} gives the desired error estimate: 
\begin{equation} 
     \| e^{n+1} \|_{-1, h}^2  \le 2 E^{n+1} 
    \le C_2 {\rm e}^{(4 C_0 + 5)T} \hat{R}_{4, \varepsilon}^* ( \dt^4 + h^8 ) .   
      \label{convergence-est-9-5}
\end{equation} 

\noindent
{\bf Recovery of the a-priori assumption (\ref{convergence-a priori-1}) }  \, \, With the convergence estimate~\eqref{convergence-est-9-5} at hand, it is obvious that the a-priori assumption~\eqref{convergence-a priori-1} could be recovered at the next time step, $t^{n+1}$, under the following constraint for $\dt$ and $h$: 
\begin{equation} 
\begin{aligned} 
  &
   C_2 {\rm e}^{(4 C_0 +5) T} \hat{R}_{4, \varepsilon}^*  
    \max ( \dt^{\frac14} , h^\frac12 ) \le 1 ,    
\\
  &
  \mbox{i.e.} \quad 
  \dt \le  ( C_2 {\rm e}^{(4 C_0 +5) T} \hat{R}_{4, \varepsilon}^*  )^{-4} , \, \, \, 
   h \le  ( C_2 {\rm e}^{(4 C_0 +5) T} \hat{R}_{4, \varepsilon}^*  )^{-2} . 
\end{aligned} 
   \label{convergence-condition-3} 
\end{equation} 
In turn, we can take $C_0^* = 2 C_0 + \frac52$,  $\hat{R}^* = C_2^\frac12 {\rm e}^{(2 C_0 + \frac52)T} ( \hat{R}_{4, \varepsilon}^* )^\frac12$. This matches the form of~\eqref{convergence-1}, in which the integer index $J_0$ could be chosen according to the form of $\hat{R}_{4, \varepsilon}^*$. 

Moreover, based on the dependence of $\hat{R}_{2, \varepsilon}^*$, $\hat{R}_{4, \varepsilon}^*$ in terms of $\varepsilon^{-1}$, the values of $J_1$ and $J_2$ could be appropriately taken. In other words, the constraints \eqref{convergence-condition-2} and \eqref{convergence-condition-3} for $\dt$ and $h$ could be converted into the form of~\eqref{convergence-condition-1}. As a result, the proof of Theorem~\ref{thm:convergence} is finished.

\begin{rem} 
A refined error estimate has been reported for a Crank-Nicolson-style, energy stable semi-discrete numerical scheme~\cite{guo21}, while the space was kept continuous. For the fully discrete Crank-Nicolson-style numerical schemes~\cite{cheng16a, diegel16, guo16}, with either finite difference, mixed finite element or Fourier pseudo-spectral spatial approximation, the analysis could be naturally extended, following similar ideas of this article. Of course, for the BDF-style, energy stable numerical schemes, with other choices of spatial discretization, the refined error analysis (with an improved convergence constant) could also be similarly derived. The technical details are expected to be very involved and are left to interested readers. 

The refined error estimate for the third and even higher order (in time) numerical schemes for the Allen-Cahn/Cahn-Hilliard equation, such as \cite{chenXW2022, cheng2022a}, will be considered in the future works. 
\end{rem}

	\section {The numerical results}
	\label{sec:numerical results}

In this section we provide a numerical accuracy check for the proposed numerical scheme~\eqref{scheme-BDF-CH-1}. The computational domain is set as $\Omega = (0,3.2)^3$, and the exact profile for the phase variable is chosen to be
	\begin{equation}
\Phi (x,y,t) =  \cos ( \frac58 \pi x) \sin( \frac58 \pi y) \cos ( \frac58 \pi z ) \cos( t) .
	\label{AC-1}
	\end{equation}
Of course, we need to add an artificial, time-dependent forcing term to make this exact profile satisfy the original equation \eqref{CH equation}. Based on this PDE with a forced term, the proposed fourth order difference scheme~\eqref{scheme-BDF-CH-1} is implemented, using the multi-grid approach. Such an iteration solver has been outlined in~\cite{guo16} for a second order difference scheme, while its extension to the fourth order finite difference approximation is straightforward.

In the accuracy check, we set the time step size as $\dt = h^2$, with $h = \frac{L}{N}$ ($L=3.2$), so that the second order temporal accuracy corresponds to the fourth order spatial accuracy. The final time set by $T=0.16$, and the surface diffusion parameter is given by $\varepsilon=1$. A sequence of spatial resolutions are taken as $N = 32:16:128$.  The expected numerical accuracy assumption $e=C ( \dt^2 + h^4)$ indicates that $\ln |e|=\ln C - 4 \ln N$. In turn, we plot $\ln |e|$ versus $\ln N$ to demonstrate the temporal convergence order. The fitted line displayed in Figure~\ref{fig1} shows an approximate slope of -3.9905, which confirms almost a perfect fourth order convergence in space.

	\begin{figure}
	\begin{center}
\includegraphics[width=3.0in]{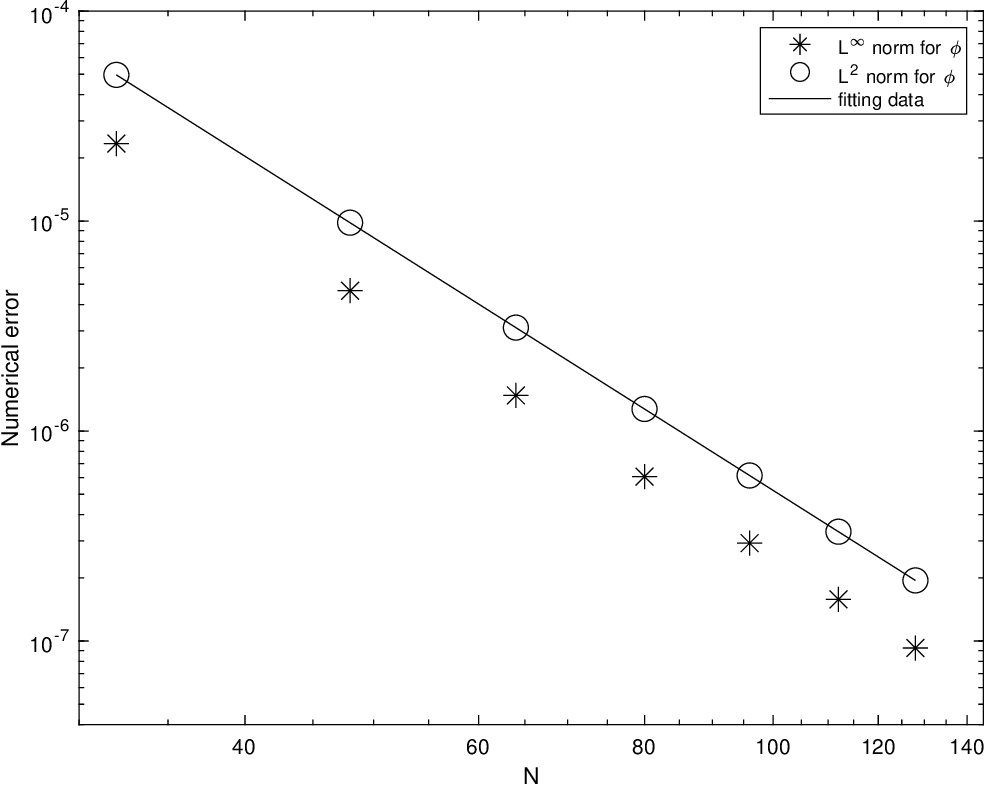}
	\end{center}
\caption{The discrete $\ell^2$ and $\ell^\infty$ numerical errors versus spatial resolution  $N$ for $N = 32:16:128$, and the time step size is set as $\dt = h^2$. The numerical results are obtained by the computation using the proposed numerical scheme~\eqref{scheme-BDF-CH-1}. The surface diffusion parameter is taken to be $\varepsilon=1$. The data lie roughly on curves $CN^{-4}$ for appropriate choices of $C$, confirming the full fourth order spatial accuracy and second order temporal accuracy.}
	\label{fig1}
	\end{figure} 
	
This accuracy check confirms the local-in-time convergence property of the fourth order finite difference numerical scheme~\eqref{scheme-BDF-CH-1}. In terms of the long time behavior of this numerical scheme, the two-dimensional simulation results of the coarsening process has been presented in~\cite{cheng2019a}, in a square domain $\Omega = (0, 12.8)^2$, with a surface diffusion coefficient $\varepsilon=0.03$. An almost perfect $t^{-1/3}$ energy dissipation law has been reported in the long time simulation, which in turn verifies the theoretical analysis provided in this article.

\section{Concluding remarks}  \label{sec:conclusion} 
A refined error analysis is presented for a modified BDF2 in time, fourth order long stencil finite difference numerical scheme to the 3-D Cahn-Hilliard equation. The Fourier analysis has to be applied to analyze the functional norms associated with the fourth order long stencil spatial approximation, and the difference between the discrete and continuous inner product has to be estimated in an accurate way. For the fully discrete numerical scheme, a uniform-in-time $H^m$ bound of the numerical solution, for any $m \ge 2$, as well as the associated $H^k$ bounds for the first and second order temporal difference stencil, namely $\| \phi^{n+1}_{\bf F} - \phi_{\bf F}^n \|_{H^k}$ and $\| \phi^{n+1}_{\bf F} - 2 \phi_{\bf F}^n + \phi_{\bf F}^{n-1} \|_{H^k}$, are needed to obtain a refined convergence constant. All these functional bounds have to be carefully derived through Sobolev estimates in the fourth order finite difference space. Certain recursive analysis has been applied in the analysis for the BDF2 temporal stencil. In the error estimate, we have to apply a spectrum estimate for the linearized Cahn-Hilliard operator, so that all the numerical error inner product terms, in the discrete $H^{-1}$ space, are analyzed in a unified way. Using this analytic approach, an application of the discrete Gronwall inequality avoids a convergence constant of the form $\mbox{exp} ( C T \varepsilon^{-m})$; instead, the constant turns out to be dependent on $\varepsilon^{-1}$ only in a polynomial order. A three-dimensional numerical example of accuracy check is also presented, which confirms the theoretical analysis in this article.

	\section{Acknowledgments}
This work is supported in part by the NSF DMS-2012269, DMS-2309548 (C.~Wang), NSFC 12001358 (Y. Yan), and NSFC 11971342, 12371401 (X.~Yue).

	\appendix

	\section{Proof of Lemma~\ref{lemma:1} }
	\label{lemma 1-proof}

	\begin{proof}
For the discrete grid function $f$ and its continuous extension $f_{\bf F}$, given by~\eqref{def:Fourier} and \eqref{def:extension}, Parseval's identity (at both the discrete and continuous levels) 
implies that
\begin{equation}  
   \nrm{f}^2_{2} = \nrm{f_{{\bf F}}}^2_{L^{2}} 
   = L^3 \sum^{K}_{\ell,m,n=-K}
|\hat{f}^N_{\ell,m,n}|^2 ,  \quad \mbox{since $h N = L$} . 
\end{equation}

For the comparison between the discrete and continuous gradient and Laplacian operators, we start with the following Fourier expansions:
	\begin{align}
(D_x f)_{i+1/2,j,k} 
=& \sum^{K}_{\ell,m,n=-K} \mu_{\ell}  \hat{f}^N_{\ell,m,n}  
 {\rm e}^{2 \pi i  ( \ell x_{i+1/2} 
  + m y_{j} + n z_{k} )/ L }  ,  \label{Lemma 1-2-0}
        \\
 (D_x^2 f)_{i,j,k} 
=& \sum^{K}_{\ell,m,n=-K} - | \mu_{\ell} |^2  \hat{f}^N_{\ell,m,n}  
 {\rm e}^{2 \pi i  ( \ell x_{i+1/2} 
  + m y_{j} + n z_{k} )/ L }  ,  \label{Lemma 1-2-0-2}
	\\
  \partial_x f_{{\bf F}} (x,y,z) =& 
\sum^{K}_{\ell,m,n=-K}  \nu_{\ell} 
 \hat{f}^N_{\ell,m,n} {\rm e}^{2 \pi i ( \ell x 
 + m y + n z )/ L } ,  \label{Lemma 1-2-1}
\\
  \mbox{with} \quad & 
\mu_{\ell} = -\frac{2 i \sin{\frac{\ell\pi h}{L}}}{h}, \quad
\nu_{\ell} = -\frac{2 i \ell\pi}{L}.  \label{Lemma 1-2-2} 
\end{align}
In turn, an application of Parseval's identity yields
\begin{eqnarray}
\nrm{D_x f}^2_2 = L^3\sum^{K}_{\ell,m,n=-K} 
 |\mu_{\ell}|^2|\hat{f}^N_{\ell,m,n}|^2, \quad 
\nrm{\partial_x f_{{\bf F}}}^2_{L^2} =
 L^3\sum^{K}_{\ell,m,n=-K}|\nu_{\ell}|^2|
  \hat{f}^N_{\ell,m,n}|^2.
\end{eqnarray}
The comparison of Fourier eigenvalues between $|\mu_{\ell}|$ and $|\nu_{\ell}|$ shows that
\begin{equation}
\frac{2}{\pi} |\nu_{\ell}| \le |\mu_{\ell}| \le |\nu_{\ell}|, 
\quad \rm{for}  \quad -K \le {\ell} \le K . \label{Lemma 1-2-3}
\end{equation}
This indicates that
\begin{equation}
\frac{2}{\pi} \nrm{\partial_x \phi_{{\bf F}}}_{L^2} 
\le \nrm{D_x \phi}_2 \le \nrm{\partial_x \phi_{{\bf F}}}_{L^2}.
\end{equation}
Similar comparison estimates can be derived in the same manner to reveal that 
\begin{equation}
\frac{2}{\pi}\|\nabla\phi_{{\bf F}}\|_{L^2} \le 
\|\nabla_{h}\phi\|_2 \le \|\nabla\phi_{{\bf F}}\|_{L^2}.  
 \label{Lemma 1-2}
\end{equation} 
This gives (\ref{lem-1-2}) in Lemma~\ref{lemma:1}, with $j=0$. It can be proved analogously that
\begin{eqnarray}
&&  ( 2 \pi^{-1} )^{2j} 
 \|\Delta^j \phi_{{\bf F}}\|_{L^2} 
  \le \|\Delta_{h}^j \phi\|_2 
  \le \|\Delta^j \phi_{{\bf F}}\|_{L^2},  \label{Lemma 1-3} \\
&&  ( 2 \pi^{-1} )^{2j+1} 
\|\nabla\Delta^j \phi_{{\bf F}}\|_{L^2} 
\le \|\nabla_{h}\Delta_{h}^j \phi\|_2 
\le \|\nabla\Delta^j \phi_{{\bf F}}\|_{L^2} ,  
  \label{Lemma 1-4}  
\end{eqnarray}
for any $j \ge 1$. 
As a result, both (\ref{lem-1-1}) and (\ref{lem-1-2}) have been established. 

In turn, estimate~\eqref{H2-est-6-7} is a direct consequence of~\eqref{Lemma 1-3}-\eqref{Lemma 1-4}, combined with the preliminary inequality~\eqref{inequality-0-2} (in Lemma~\ref{lem: inequality}), as well as the elliptic regularity at the continuous level: 
\begin{equation} 
\begin{aligned}  
 \|  \Delta_{h, (4)} f \|_2  \le & \frac43 \| \Delta_h f \|_2 
  \le \frac43 \| \Delta f_{\bf F} \|  
  \le \frac43 C_1^* \| \Delta^2 f_{\bf F} \|  
\\
  \le & \frac43 C_1^* (D_4)^{-1} \| \Delta_h^2 f \|_2
  \le \frac43 C_1^* (D_4)^{-1} \| \Delta_{h, (4)}^2 f \|_2  ,  
\end{aligned} 
    \label{Lemma 1-4-2}  
\end{equation}
in which $C_1^*$ is denoted as the elliptic regularity constant at the continuous level. Therefore, \eqref{H2-est-6-7}  has been proven, by taking $C_1 = \frac43 C_1^* (D_4)^{-1}$.

To prove \eqref{lem-1-3}, we first observe that $\int_\Omega f_{\bf F} , d \x = | \Omega | \cdot \overline{f} = 0$, since $f_{\bf F} = 0$. In turn, a careful calculation reveals the discrete Fourier expansion for $( - \Delta_{h, (4)})^{-1} f$ and the continuous Fourier expansion for 
$( - \Delta)^{-1} f_{\bf F}$
	\begin{align} 
	  &
  ( - \Delta_{h, (4)})^{-1} f_{i,j,k} = 
  \sum^{K}_{\ell,m,n=-K, (\ell, m,n) \ne \0}
  ( \lambda_{\ell, m, n}^{(4)} )^{-1} \hat{f}^N_{\ell,m,n} 
  {\rm e}^{2 \pi i ( \ell x_{i} + m y_{j} + n z_{k} )/ L }  , 
   \label{Lemma 1-5-1} 
\\
  &
  ( - \Delta)^{-1} f_{{\bf F}}(x,y,z) 
  = \sum^{K}_{\ell,m,n=-K, (\ell, m,n) \ne \0} 
  \Lambda_{\ell, m, n}^{-1} \hat{f}^N_{\ell,m,n} 
  {\rm e}^{2 \pi i ( \ell x + m y + n z )/ L }  , 
  \label{Lemma 1-5-2} 
\\
  & 
    \mbox{with} \quad 
  \lambda_{\ell, m, n}^{(4)} = \lambda_\ell^{(4)} + \lambda_m^{(4)} 
  + \lambda_n^{(4)} ,  \quad  \lambda_k^{4)}  =\lambda_k + \frac{h^2}{12} \lambda_k^2 , 
  \, \, \, \lambda_k =  \frac{4 \sin^2 
  {\frac{k \pi h}{L}}}{h^2} , \, \, ( 0 \le k \le K ) , 
   \label{Lemma 1-5-3}  
\\
  & 
      \qquad  \, \, 
  \Lambda_{\ell, m, n} = \Lambda_\ell + \Lambda_m 
  + \Lambda_n ,  \quad  
  \Lambda_k = \frac{4 k^2 \pi^2}{L^2} , \, \, ( 0 \le k \le K ) .  
   \label{Lemma 1-5-4} 
	\end{align}
with the zero vector defined as $\0 = (0,0,0)$. Furthermore, an application of the Parseval equality to the discrete Fourier expansion for $ \nabla_{h, (4)} ( ( - \Delta_{h, (4)})^{-1} f )$ 
and the continuous Fourier expansion for $\nabla \left( ( - \Delta)^{-1} f_{\bf F} \right)$ implies that 
	\begin{align} 
	  &
  \| f \|_{-1, h}^2  = 
  \nrm{ \nabla_{h, (4)} ( - \Delta_{h, (4)})^{-1} f }_2^2 = 
  L^3 \sum^{K}_{\ell,m,n=-K, (\ell, m,n) \ne \0}
  ( \lambda_{\ell, m, n}^{(4)} )^{-1}  
  \left| \hat{f}^N_{\ell,m,n} \right|^2 , 
   \label{Lemma 1-6-1} 
\\
  &
 \| f_{\bf F} \|_{H^{-1}}^2 
 =  \nrm{ \nabla ( - \Delta)^{-1} f_{{\bf F}} }^2  
  = L^3 \sum^{K}_{\ell,m,n=-K, (\ell, m,n) \ne \0} 
  \Lambda_{\ell, m, n}^{-2} 
  ( | \nu_{\ell} |^2 + | \nu_m |^2 + | \nu_n |^2 )  
  \left| \hat{f}^N_{\ell,m,n}  \right|^2  \nonumber 
\\
  &  \qquad \qquad \qquad \qquad \qquad \qquad 
   = L^3 \sum^{K}_{\ell,m,n=-K, (\ell, m,n) \ne \0} 
  \Lambda_{\ell, m, n}^{-1} 
  \left| \hat{f}^N_{\ell,m,n}  \right|^2 , 
  \label{Lemma 1-6-2} 
	\end{align}
based on the following fact 
\begin{equation} 
  | \nu_{\ell} |^2 + | \nu_m |^2 + | \nu_n |^2 
  = \Lambda_{\ell, m, n} .  
\end{equation} 
Meanwhile, the eigenvalue comparison estimate \eqref{Lemma 1-2-3} indicates that 
\begin{eqnarray} 
  \frac{4}{\pi^2} \Lambda_{\ell, m, n} 
  \le \lambda_{\ell, m, n}^{(4)} \le \Lambda_{\ell, m, n} , 
  \quad \forall -K \le \ell, m, n \le K . 
  \label{Lemma 1-8} 
	\end{eqnarray}
Its combination with \eqref{Lemma 1-6-1} and \eqref{Lemma 1-6-2} yields 
	\begin{eqnarray} 
  \nrm{ f_{\bf F} }_{H^{-1} }^2 
   \le \nrm{ f }_{-1, h}^2 
   \le \frac{\pi^2}{4} \nrm{ f_{\bf F} }_{H^{-1} }^2 . 
  \label{lemma 1-9} 
	\end{eqnarray}
Therefore, \eqref{lem-1-3} has been proven, by taking $D_{-1} = \frac{\pi}{2}$.  

Similar techniques can be used to establish~\eqref{lem-1-4}. We have the discrete Fourier expansion for $\Delta_{h, (4)} f$ and the continuous Fourier expansion for 
$\Delta f_{\bf F}$
	\begin{align} 
	  &
   \Delta_{h, (4)} f_{i,j,k} =  - \sum^{K}_{\ell,m,n=-K}
  \lambda_{\ell, m, n}^{(4)} \hat{f}^N_{\ell,m,n} 
  {\rm e}^{2 \pi i ( \ell x_{i} + m y_{j} + n z_{k} )/ L }  ,   
  \quad \mbox{so that} 
   \label{Lemma 1-10-1} 
\\
  &   
  ( \Delta_{h, (4)} f)_{\bf F} (x,y,z) = 
  - \sum^{K}_{\ell,m,n=-K}
  \lambda_{\ell, m, n}^{(4)} \hat{f}^N_{\ell,m,n} 
  {\rm e}^{2 \pi i ( \ell x + m y + n z )/ L } , 
  \label{Lemma 1-10-2} 
\\
  &
    \Delta f_{{\bf F}}(x,y,z) 
  = - \sum^{K}_{\ell,m,n=-K} 
  \Lambda_{\ell, m, n} \hat{f}^N_{\ell,m,n} 
  {\rm e}^{2 \pi i ( \ell x + m y + n z )/ L }  .  
  \label{Lemma 1-10-3} 
	\end{align}
Estimate (\ref{lem-1-4}), with $m=0$, is a direct consequence of the eigenvalue comparison estimate~\eqref{Lemma 1-8}. To make a comparison between their $H^m$ norm, we observe 
the following Fourier expansion of their corresponding 
$\partial_x$ derivatives: 
	\begin{align} 
	 &  
  \partial_x ( \Delta_{h, (4)} f)_{\bf F} (x,y,z) = 
  - \sum^{K}_{\ell,m,n=-K}
  \nu_\ell \lambda_{\ell, m, n}^{(4)} \hat{f}^N_{\ell,m,n} 
  {\rm e}^{2 \pi i ( \ell x + m y + n z )/ L } , 
   \label{Lemma 1-11-1} 
\\
  &
  \partial_x \Delta f_{{\bf F}}(x,y,z) 
  = - \sum^{K}_{\ell,m,n=-K} 
  \nu_\ell \Lambda_{\ell, m, n} \hat{f}^N_{\ell,m,n} 
  {\rm e}^{2 \pi i ( \ell x + m y + n z )/ L }  ,   
  \label{Lemma 1-11-2} 
	\end{align}
and an application of Parseval equality implies that 
\begin{align} 
  &  
  \nrm{ \partial_x ( \Delta_h f)_{\bf F} }^2  
  = L^3 \sum^{K}_{\ell,m,n=-K}
  | \nu_\ell |^2 ( \lambda_{\ell, m, n}^{(4)} ) ^2 
  \left| \hat{f}^N_{\ell,m,n} \right|^2 ,  
   \label{Lemma 1-12-1}  
\\
  &
  \nrm{ \partial_x \Delta f_{{\bf F}} }^2  
  = L^3  \sum^{K}_{\ell,m,n=-K} 
  | \nu_\ell |^2 \Lambda_{\ell, m, n}^2 
  \left| \hat{f}^N_{\ell,m,n}  \right|^2 .    
  \label{Lemma 1-12-2} 
	\end{align}
By the estimate eigenvalue comparison estimate \eqref{Lemma 1-8}, we conclude that 
\begin{equation}   
  \nrm{ \partial_x ( \Delta_h f)_{\bf F} }^2  
  \le \nrm{ \partial_x \Delta f_{{\bf F}} }^2  .    
  \label{Lemma 1-13} 
	\end{equation}
Similar estimates can be derived for other partial derivatives, and higher order derivatives. 
Therefore, (\ref{lem-1-4}) is valid for any $m \ge 0$. The proof of Lemma~\ref{lemma:1} is complete. 
	\end{proof}

	\section{Proof of Lemma~\ref{lemma:2} }
	\label{lemma 2-proof}

	\begin{proof}
Based on the Fourier expansions~\eqref{Lemma 1-2-0}-\eqref{Lemma 1-2-2}, we apply the Parseval equality and get 
\begin{align} 
  & 
  \nrm{ D_x f }_2^2  
= L^3 \sum^{K}_{\ell,m,n=-K}  | \mu_{\ell} |^2 
\cdot \left| \hat{f}^N_{\ell,m,n} \right|^2 ,   \quad 
   \| D_x^2 f \|_2^2  
= L^3 \sum^{K}_{\ell,m,n=-K}  | \mu_{\ell} |^4 
\cdot \left| \hat{f}^N_{\ell,m,n} \right|^2
 \label{Lemma 2-1-1}
	\\
	&
  \nrm{ \partial_x f_{{\bf F}} }^2 = 
 L^3 \sum^{K}_{\ell,m,n=-K}  | \nu_{\ell} |^2 
 \cdot \left| \hat{f}^N_{\ell,m,n} \right|^2 ,  
   \label{Lemma 2-1-2}
\\
  & 
  \mbox{so that} \quad 
  \nrm{ \partial_x f_{{\bf F}} }^2 
   - \nrm{ {\cal D}_{x,(4)} f }_2^2  = 
 L^3 \sum^{K}_{\ell,m,n=-K}  
  \left( | \nu_{\ell} |^2 - ( | \mu_{\ell} |^2 + \frac{h^2}{12} | \mu_\ell |^4 ) \right) 
    \left| \hat{f}^N_{\ell,m,n} \right|^2 .  
    \label{Lemma 2-1-3} 
\end{align}
Furthermore, the following estimates are available: 
\begin{align} 
  & 
  \sin{\frac{\ell\pi h}{L}}  
  = \frac{\ell\pi h}{L} - \frac16 \Big( \frac{\ell\pi h}{L} \Big)^3 
  - \frac{\cos \eta}{120} \Big( \frac{\ell\pi h}{L} \Big)^5 , 
  \quad \mbox{$\eta \in (0, \frac{\pi}{2}) $} ,  
  \label{Lemma 2-2-1}  
\\
  & 
  \sin^2 \Big( \frac{\ell\pi h}{L} \Big)   
  = \Big( \frac{\ell\pi h}{L} \Big)^2 - \frac13 \Big( \frac{\ell\pi h}{L} \Big)^4 
  +  ( \frac{1}{36} + \frac{\cos \eta}{60} ) \Big( \frac{\ell\pi h}{L} \Big)^6 
  + \xi_1 \Big( \frac{\ell\pi h}{L} \Big)^8  , 
  \quad \mbox{$| \xi_1 | \le \frac{1}{360}$} ,  
  \label{Lemma 2-2-2}     
\\
  & 
   1 + \frac{h^2}{12} | \mu_{\ell} |^2 
   = 1 + \frac13 \sin^2 \Big( \frac{\ell\pi h}{L} \Big)   
  = 1 + \frac13 \Big( \frac{\ell\pi h}{L} \Big)^2 + \xi_2 \Big( \frac{\ell\pi h}{L} \Big)^4 , 
  \quad \mbox{$- \frac{1}{9} \le \xi_2  \le - \frac{1}{12}$} ,  
  \label{Lemma 2-2-3}        
\\
  & 
  \mbox{so that} \, \, \,   0 \le | \nu_{\ell} |^2 
   - | \mu_{\ell} |^2 ( 1 + \frac{h^2}{12} | \mu_{\ell} |^2 ) 
  =  ( \frac13 - \frac{\cos \eta}{15} - 4 \xi_2 ) h^4 \Big( \frac{\ell\pi}{L} \Big)^6 
  + \xi_3 h^6 \Big( \frac{\ell\pi}{L} \Big)^8  \, \, \, (\mbox{$ - \frac29 \le \xi_3 \le 0$}) 
  \nonumber 
\\
  &  \qquad \qquad 
    \le h^4 \Big( \frac{\ell\pi}{L} \Big)^6  = \frac{1}{64} h^4  
     \left( \frac{2 \ell\pi}{L} \right)^6  , 
     \label{Lemma 2-2-4} 
\end{align} 
in which the eigenvalue representation formula~\eqref{Lemma 1-2-2}  was recalled, and a Taylor expansion was performed in~\eqref{Lemma 2-2-1}. Going back~\eqref{Lemma 2-1-3}, we arrive at 
\begin{eqnarray}
  0 \le \nrm{ \partial_x f_{{\bf F}} }^2 
   - \nrm{ {\cal D}_{x, (4)} f }_2^2  \le  
 \frac{L^3}{64} \sum^{K}_{\ell,m,n=-K}  
   h^4 \left( \frac{2 \ell\pi}{L} \right)^6 
    \left| \hat{f}^N_{\ell,m,n} \right|^2  
    = \frac{h^4}{64} \nrm{ \partial_x^3 f_{{\bf F}} }^2 .   
    \label{Lemma 2-3} 
\end{eqnarray}
Two other estimates could be derived in a similar way:  
\begin{eqnarray}
  0 \le \nrm{ \partial_y f_{{\bf F}} }^2 
   - \nrm{ {\cal D}_{y, (4)} f }_2^2    
  \le \frac{h^4}{64} \nrm{ \partial_y^3 f_{{\bf F}} }^2 , \, \,  
     0 \le \nrm{ \partial_z f_{{\bf F}} }^2 
   - \nrm{ {\cal D}_{z, (4)} f }_2^2    
  \le \frac{h^4}{64} \nrm{ \partial_z^3 f_{{\bf F}} }^2 . 
    \label{Lemma 2-4} 
\end{eqnarray}
A combination of \eqref{Lemma 2-3} and \eqref{Lemma 2-4} results in 
\begin{eqnarray}
  0 \le \nrm{ \nabla f_{{\bf F}} }^2 
   - \nrm{ \nabla_{h, (4)} f }_2^2      
  \le \frac{h^4}{64} \nrm{ f_{{\bf F}} }_{H^3}^2 . 
    \label{Lemma 2-5} 
\end{eqnarray}
The proof of Lemma~\ref{lemma:2} is complete. 
\end{proof}

 	\section{Proof of Lemma~\ref{lemma:3} }
	\label{lemma 3-proof}

	\begin{proof}
(1) In addition to (\ref{def:Fourier}) and (\ref{def:extension}), we set the discrete Fourier expansion 
for $g$ and its continuous extension given by 
	\begin{align}
g_{i,j,k} =& \sum^{K}_{\ell,m,n=-K}
\hat{g}^N_{\ell,m,n} {\rm e}^{2 \pi i ( \ell x_{i} 
 + m y_{j} + n z_{k} )/ L }  ,
   \label{Lemma 3-1-1}
\\
   {\mathbf g} (x,y,z) 
   =& g_{{\bf F}}(x,y,z) = \sum^{K}_{\ell,m,n=-K} 
 \hat{g}^N_{\ell,m,n} {\rm e}^{2 \pi i 
   ( \ell x + m y + n z )/ L }  .  
   \label{Lemma 3-1-2}
	\end{align}
In turn, we assume the Fourier expansion for the product function ${\mathbf f} \cdot {\mathbf g}$ as 
\begin{equation}
   ( {\mathbf f} \cdot {\mathbf g} ) (x,y,z) 
   = \sum^{2K}_{\ell,m,n=-2K} 
 \hat{h}^N_{\ell,m,n} {\rm e}^{2 \pi i 
   ( \ell x + m y + n z )/ L }  .  
   \label{Lemma 3-2}
	\end{equation} 
In particular, it is observed that ${\mathbf f} \cdot {\mathbf g} \in {\cal B}^{2K}$. Consequently, the discrete product function $f \cdot g$ turns out to be the projection of ${\mathbf f} \cdot {\mathbf g}$ at the numerical grid points: 
    \begin{equation} 
  ( f \cdot g )_{i,j,k} 
  = ( {\mathbf f} \cdot {\mathbf g} ) (x_i, j_j, z_k ) 
  = {\cal I}_N \left( {\mathbf f} \cdot {\mathbf g} \right) 
  (x_i, j_j, z_k ) .   \label{Lemma 3-3}
	\end{equation} 
A more careful expansion shows that 
\begin{equation} 
  \overline{ f \cdot g } 
  = \frac{1}{|\Omega|}  \int_\Omega 
   {\cal I}_N ( {\mathbf f} \cdot {\mathbf g} ) d \x 
  = \frac{1}{|\Omega|}  \int_\Omega 
    {\mathbf f} \cdot {\mathbf g}  d \x , 
    \label{Lemma 3-4}
	\end{equation}
which is equivalent to \eqref{lem-3-1}. In more detail, the first step comes from the fact that ${\cal I}_N {\mathbf f} \cdot {\mathbf g} \in {\cal B}^K$, and the second step is based on the fact that, there is no aliasing error on the mode of $(\ell, m, n) = \0$, between ${\mathbf f} \cdot {\mathbf g} \in {\cal B}^{2K}$ and its projection onto ${\cal B}^K$. 

(2) In the general case, we note that $f$ and $g$ are discrete interpolations of ${\cal I}_N {\mathbf f} \in {\cal B}^K$ and ${\cal I}_N {\mathbf g} \in {\cal B}^K$. By \eqref{lem-3-1}, we arrive at 
    \begin{eqnarray} 
   \left| \left\langle f , g \right\rangle  
  - \left( {\mathbf f} , {\mathbf g}  \right)  \right| 
   &=&  \left| \left( {\cal I}_N {\mathbf f} , 
   {\cal I}_N {\mathbf g} \right)   
  - \left( {\mathbf f} , {\mathbf g}  \right)  \right| 
  \le \left|  \left( {\cal I}_N {\mathbf f}  - {\mathbf f} , 
   {\cal I}_N {\mathbf g} \right)  \right| 
  + \left|  \left( {\mathbf f} , 
  {\cal I}_N {\mathbf g} - {\mathbf g}  \right)  \right| 
  \nonumber  
\\
  &\le& 
  \nrm{ {\cal I}_N {\mathbf f}  - {\mathbf f} }  
   \cdot \nrm{ {\cal I}_N {\mathbf g}  }  
  + \nrm{ {\mathbf f} }  \cdot  
  \nrm{ {\cal I}_N {\mathbf g} - {\mathbf g} }  \nonumber  
\\
  &\le& C  h^8 \left( 
  \nrm{ {\mathbf f} }_{H^8}  \cdot \nrm{ {\mathbf g} }_{H^2}
   + \nrm{ {\mathbf f} }_{H^2}  
   \cdot \nrm{ {\mathbf g} }_{H^8}  \right) ,  
  \label{Lemma 3-5}
	\end{eqnarray}
which gives (\ref{lem-3-2}), with the Fourier spectral 
interpolation approximation (\ref{spectral-approximation-2}) 
applied at the last step. The proof of Lemma~\ref{lemma:3} 
is complete. 
\end{proof}

	\bibliographystyle{plain}
	\bibliography{revision1.bib}

	\end{document}